\documentclass[12pt,twoside]{article}
\usepackage[dutch,english]{babel}
\usepackage{amsmath, amssymb, amsthm, graphicx, makeidx,esint,makeidx, enumerate}
\usepackage{bbold,eurosym}
\usepackage{yhmath}
\usepackage{fancyhdr,rotating,hyperref}
\usepackage[top=4.3cm, bottom=3.4cm, left=3.5cm, right=3.5cm]{geometry}

\newcommand*{\doi}[1]{\href{http://dx.doi.org/#1}{doi: #1}}

\theoremstyle{plain}

\newtheorem{theorem}{Theorem}[section]
\newtheorem{lemma}[theorem]{Lemma}

\newtheorem{corollary}[theorem]{Corollary}

\newtheorem{definition}[theorem]{Definition}

\newtheorem{notation}[theorem]{Notation}
\newtheorem{conjecture}[theorem]{Conjecture}

\def\wpo{\mathbb{wpo}}

\def\X{\mathbb{X}}
\def\Y{\mathbb{Y}}
\def\Z{\mathbb{Z}}

\DeclareMathAlphabet{\roundT}{OMS}{cmsy}{m}{n}
\newcommand{\T}{\roundT{T}}

\title{An order-theoretic characterization of the Howard-Bachmann-hierarchy}

\author{Jeroen Van der Meeren\thanks{jeroen.vandermeeren@ugent.be, 
Department of Mathematics, Ghent University, Krijgslaan 281, 9000 Gent, Belgium}, Michael Rathjen\thanks{rathjen@maths.leeds.ac.uk, Department of Pure Mathematics, University of Leeds, Leeds LS2 9JT, England}, Andreas Weiermann\thanks{andreas.weiermann@ugent.be, Department of Mathematics, Ghent University, Krijgslaan 281, 9000 Gent, Belgium}
}
\date{}

\begin{document}

\maketitle
\thispagestyle{empty}

\begin{abstract}



In this article we provide an intrinsic characterization of the famous Ho\-ward-Bach\-mann ordinal in terms of a natural well-partial-ordering by showing that this ordinal can be realized as a maximal order type of a class of generalized trees with respect to a homeomorphic embeddability relation.
We use our calculations to draw some conclusions about some corresponding subsystems of second order arithmetic. All these subsystems deal with versions of light-face $\Pi^1_1$-comprehension.
\end{abstract}

\section{Introduction}\label{intro}

The famous Howard-Bachmann ordinal $\eta_0$ (which in the literatur is also denoted by $\psi\varepsilon_{\Omega+1}$, $\vartheta\varepsilon_{\Omega+1}$, $\theta\varepsilon_{\Omega+1}0$, $d\varepsilon_{\Omega+1}$) belongs to the most well-established arsenal of proof-theoretic ordinals of natural theories for developing significant parts of (impredicative) mathematics. Of course $\eta_0$ is much bigger than $\varepsilon_0$, the proof-theoretic ordinal of first order Peano arithmetic, and it is also bigger than $\Gamma_0$, the proof-theoretic ordinal of predicative analysis. The ordinal $\eta_0$ is the proof-theoretic ordinal of the first order theory $ID_1$, which extends $PA$ by schemes for smallest fixed points of non-iterated positive inductive definitions. The ordinal $\eta_0$ is also the proof-theoretic ordinal of the theory $KP\omega$ which formalizes an admissible universe containing $\omega$, and $\eta_0$ is also the proof theoretic ordinal of $ACA_0+(\Pi^1_1$--\,$CA)^-$ which formalizes lightface $\Pi^1_1$-comprehension and of the theory $RCA_0+(BI)$ which extends $RCA_0$ by a scheme of bar induction.

All these theories are considered to be impredicative. For example the theory $ACA_0+(\Pi^1_1$--\,$CA)^-$ allows the formation of new sets of natural numbers by using a comprehension formula which may contain a set quantifier ranging over the set which is just defined using the comprehension under consideration. An ordinal analysis for each of these impredicative theories turned out to be difficult. In technical terms this is usually reflected by proof calculi with rules where the complexity of the antecedent is greater than the complexity of the succedent. Typically this is a stumbling block for cut elimination and only very sophisticated methods like Buchholz' operator-controlled derivations can circumvent it.

It is important to notice that proof-theoretic ordinals are more than just shere set-theoretic ordinals. They usually come equipped with a first order structure with certain built-in functions which generate the ordinal in question in some natural and perspicuous way. On a more combinatorial level such an ordinal is then usually re\-presented by a certain primitive recursive set of terms with a primitive recursive well-ordering relation. A precise description on what the characteristics of a proof-theoretic ordinal are runs under the "natural well-ordering problem" which is known to be notoriously difficult. Investigations on proof-theoretic ordinals have been undertaken by many people. To name a few: Aczel, Arai, Bachmann, Bridges, Buchholz, Feferman, Gerben, Girard, Gordeev, Isles, J\"ager, Kino, Levitz, Okada, Pfeiffer, Pohlers, Probst, Rathjen, Sch\"utte, Setzer, Strahm, Takeuti, Veblen and Weiermann.

\medskip 

A very important aspect of investigations on proof-theoretic ordinals goes back to Diana Schmidt who first recognized the connection between these ordinals and order-theoretic properties of the functions generating them. She characterized completely the order types which could be generated from the ordinal $0$ by applying a monotonic increasing function. A monotonic increasing binary function generates out of the singleton set containing the ordinal $0$ no order type larger than $\varepsilon_0$. Functions of bigger arities produce easily ordinals bigger than $\Gamma_0$ and in fact ordinals of size comparable to the small Veblen ordinal $\vartheta \Omega^\omega$ but by no means an ordinal which is of size comparable to $\eta_0$.

Diana Schmidt moreover showed that studying bounds on closure ordinals can best be achieved by determining maximal order types of well-partial-orderings which reflect monotonicity properties of the functions in question. With regard to this research program she classified maximal order types for various classes of labelled trees. The ordinals obtained in this way are all around the small Veblen ordinal $\vartheta \Omega^\omega$ and way below $\eta_0$ and it was for some time not clear whether $\eta_0$ can be characterized in terms of closure ordinals.

In Weiermann \cite{boundsfortheclosureordinalsofessentiallymonotonicincreasingfunctions}, it was shown that $\eta_0$ could indeed be characterized as a closure ordinal of so-called essentially monotonic increasing functions. Since then it has been open whether a corresponding order-theoretic characterization in terms of maximal order types is possible. Weiermann's proof made essential use of the linearity of ordinals and did not generalize to partially ordered structures.

\medskip 

Extending Schmidt's work in \cite{ordertheoreticcharacterizationofschutteveblen}, the third author provided in a first step, an order-theoretic characterization for the large Veblen ordinal $\vartheta \Omega^\Omega$.
Quite recently, the authors of this paper were able to provide in \cite{WellpartialorderingsandthebigVeblennumber} much more convincing methods and results which were suitable for being extended to larger ordinals as well. This recent approach is already far reaching but still misses essential ingredients for a order-theoretic characterization of $\eta_0$. This paper will contribute to this problem.

\medskip

It should also be noticed that H. Friedman defined in 1985 tree-embeddability relations with a so-called gap-condition which generated ordinals of size $\psi\Omega_\omega$, an ordinal which is much bigger than $\eta_0$. So in principal it seemed plausible that it is possible to single out a natural subordering of Friedman's ordering which would match with $\eta_0$.
This paper provides also a positive answer to this challenge: the defined well-partial-order of maximal order type $\eta_0$ can be seen as a natural subordering of Friedman's tree-ordering.

Moreover, we believe that our analysis will be a starting point for classifying
the maximal order types of the full Friedman's gap-ordering on trees. We expect that analyzing Friedman's embeddability relations will be rather hard and difficult and we hope that the result provided in this paper
yields a roadmap for a more general result. We confine ourselves to $\eta_0$ since this ordinal is somehow the first serious step into impredicativity. To elaborate a little bit more on Friedman's gap-ordering: we believe that the maximal order type of Friedman's trees on $n$ labels can be described using the $n^{\text{th}}$ regular cardinal number $\Omega_n$. Hence, the uncountable cardinal numbers play a very important role in classifying the strength of Friedman's well-partial-order.

\medskip

In section \ref{sec:Preliminaries}, we give some preliminaries that are needed for later sections. In section \ref{sec:an order-theoretic approach of the Howard-Bachmann ordinal number}, we yield a well-partial-order of maximal order type $\eta_0$, that can be seen as a natural subordering of Friedman's tree-ordering. Section \ref{sec:Tree-structures below the Howard-Bachmann ordinal number} is needed to obtain the results in section \ref{sec:an order-theoretic approach of the Howard-Bachmann ordinal number}. In the last two sections, section \ref{sec:the proof-theoretical ordinals of RCA0 + Pi11CA^- and RCA0* + Pi11CA^-} and \ref{sec:independence results}, results from a pure proof-theoretical point of view are studied. We determine bounds on the proof-theoretical ordinals of theories corresponding to our well-partial-orders, from which we obtain unprovability results concerning the well-partial-orderedness of these partial orders. All these theories deal with versions of light-face $\Pi^1_1$-comprehension.



\section{Preliminaries}\label{sec:Preliminaries}
\subsection{Well-partial-orderings}\label{sec:Well-partial-orderings}

Well-partial-orderings are ordered structures that are used in different fields of mathematics. For example in Gr\"obnerbases \cite{Grobnerbases} and rewrite theory \cite{provingterminationfortermrewritingsystems}. Moreover, they are well-known objects among logicians. 
Well-partial-orderings can be seen as ge\-neralizations of well-orderings, an important notion used in ordinal analysis \cite{prooftheoryofimpredicativesubsystemsofanalysis,Pohlersboek1,Pohlersboek2}. They are the underlying concepts of the theorem of Higman \cite{higman}, the theorem of Kruskal \cite{Kruskal}, Fra\"iss\'e's order type conjecture \cite{laver} and Friedman's gap-embeddability relation on trees \cite{simpsonfinitetrees}. In \cite{rathjenweiermann}, the second and third author did a complete proof-theoretical analysis of the theorem of Kruskal.

\begin{definition} A \textbf{well-partial-ordering} (hereafter $\wpo$) is a partial ordering $(X,\leq_X)$ such that for every infinite sequence $(x_i)_{i=1}^{+\infty}$ of elements in $X$, there exists two indices $i$ and $j$ such that $i<j$ and $x_i \leq_X x_j$. We denote the $\wpo$ $(X,\leq_X)$ by $X$ if the ordering is clear from the context.
\end{definition}

So a $\wpo$ is a well-founded partial ordering that does not admit an infinite antichain. Equivalently, it is a partial order such that every extension is well-founded. In the literature, one is more familiar with the notion of a well-quasi-order. That is a quasi-order (no antisymmetry) with the same distinctive property. However, by canceling out an obvious equivalence relation, one gets a well-partial-order. Therefore, we can restrict ourselves only to $\wpo$'s. The interested reader can read more about well-partial-ordering in \cite{dejonghandparikh}, the most ground-breaking paper on this subject.

\begin{definition} The \textbf{maximal order type} of the $\wpo$ $(X,\leq_X)$ is defined as 
\begin{align*}
\text{$\sup\{\alpha$: }& \text{$\preceq$ is an extension of $\leq_X$, $\preceq$ is a well-ordering on $X$}\\
&\text{and $otype(X,\preceq)=\alpha\}$}.
\end{align*}
 We denote this ordinal as $o(X,\leq_X)$ or as $o(X)$ if the ordering on $X$ is clear from the context. 
\end{definition}

The maximal order type of a $\wpo$ is an important characteristic of that $\wpo$. E.g. one can use it in determining the exact proof-theoretical strength of the $\wpo$ under consideration. In \cite{dejonghandparikh}, it is proved that this supremum is actually a maximum, meaning that every $\wpo$ $X$ has at least one maximal linear extension.

\begin{theorem}[de Jongh and Parikh \cite{dejonghandparikh}]\label{dejonghandparikh1} Assume that $(X,\leq_X)$ is a $\wpo$. Then there exists a well-ordering $\preceq$ on $X$ which is an extension of $\leq_X$ and $otype(X,\preceq) = o(X,\leq_X)$.
\end{theorem}

In this paper, we are interested in studying the maximal order type of specific $\wpo$'s. The technique of reifications is very useful in obtaining upper bounds \cite{schuttesimpson} and lower bounds can be acquired from finding linearizations. Our technique will use the concept of left-sets $L(x)$ and quasi-embeddings. Of course, this is also interwoven with the previous mentioned techniques.

\begin{definition} Let $(X,\leq_X)$ be a partial order and $x \in X$. Define the left set $L_X(x)$ as the set $\{y \in X : x \not\leq_X y\}$ and $l_X(x):=o(L_X(x))$. We skip the subscript $X$ if this is clear from the context.
\end{definition}

\begin{theorem}[de Jongh and Parikh \cite{dejonghandparikh}]\label{dejonghandparikh2} Assume that $X$ is a partial ordering. If $L(x)$ is a $\wpo$ for every $x\in X$, then $X$ is a $\wpo$. (The converse is trivially true.) In this case, $o(X) = \sup\{l(x)+1: x \in X\}$.
\end{theorem}
 
Using this result, it can be easily seen that the maximal order type is equal to the height of the root in the tree of finite bad sequences. If one would have that $l(x) < \alpha$, for every $x \in X$, then $o(X) \leq \alpha$. Therefore, we are really interested in characterizing the left-sets of a $\wpo$. This, in combination with the notion of quasi-embeddings, will be the most important building blocks of the proofs in sections \ref{sec:Tree-structures below the Howard-Bachmann ordinal number} and \ref{sec:an order-theoretic approach of the Howard-Bachmann ordinal number}.
%

\begin{definition} Let $X$ and $Y$ two posets. A map $e:X \to Y$ is called a \textbf{quasi-embedding} if for all $x,x' \in X$ with $e(x)\leq_Y e(x')$ we have $x \leq _X x'$. 
\end{definition}

\begin{lemma}\label{quasi-embedding} If $X$ and $Y$ are posets and $e: X \to Y$ is a quasi-embedding and $Y$ is a wpo, then $X$ is a wpo and $o(X)\leq o(Y)$.
\end{lemma}


\subsection{Constructions on Well-partial-orderings}\label{sec:constructions on well-partial-orderings}

If we have $\wpo$, one can construct plenty of other well-partial-orderings from it. The two most important examples are disjoint unions and products.

\begin{definition} Let $X_0$ and $X_1$ be two partial orders. Define the \textbf{disjoint union} $X_0 + X_1$ as the set $\{(x,0): x \in X_0\}\cup \{(y,1): y \in X_1\}$ with the following ordering: 
\[(x,i) \leq (y,j) \Leftrightarrow i=j \text{ and } x \leq_{X_i} y.\]
We notate the element $(x,i)$ as $x$ if it is clear from the context in which set $x$ lies in. 
Define the \textbf{cartesian product} $X_0 \times X_1$ as the set $\{(x,y): x  \in X_0, y \in X_1\}$ with the following ordering:
\[(x,y)\leq (x',y') \Leftrightarrow x\leq_{X_0} x' \text{ and } y\leq_{X_1} y'.\]
With $X^n$ we denote the partial ordering $X \times \dots \times X$, where $X$ occurs $n$ times.
\end{definition}

In \cite{higman}, Higman studied one of the most well-known constructor on well-partial-orderings. It is in some sense a cornerstone in the theory of $\wpo$'s.

\begin{definition} Let $(X,\leq_X)$ be a partial order. Define $(X^*,\leq_X^*)$ as the partial ordering on the set $X^*$ of \textbf{finite sequences over $X$} ordered by
\begin{align*}
&(x_1,\dots,x_n) \leq_X^* (y_1,\dots,y_m)  \\
\Leftrightarrow& (\exists 1\leq i_1<\dots < i_n \leq m )(\forall j \in \{1,\dots,n\})  (x_{j} \leq_X y_{i_j}).
\end{align*}
We notate this partial ordering also as $(X^*,\leq^*)$ or even as $X^*$.
\end{definition}

In the next theorem, we state the connection of the maximal order type of these constructors on well-partial-orderings $X$ and the original order type $o(X)$.

\begin{theorem}[de Jongh and Parikh\cite{dejonghandparikh}, Schmidt\cite{dianaschmidt}]
\label{maximal order type sum product and higman}If $X_0$, $X_1$ and $X$ are $\wpo$'s, then $X_0 + X_1$, $X_0 \times X_1$ and $X^*$ are still $\wpo$'s, and
\begin{align*}
o(X_0 + X_1) &= o(X_0)\oplus o(X_1),\\
o(X_0 \times X_1) &= o(X_0) \otimes o(X_1),
\end{align*}
where $\oplus$ and $\otimes$ is the natural sum and product between ordinals, and 
\[o(X^*) = \left\{ 
\begin{array}{ll}
\omega^{\omega^{o(X)-1}} & \text{if $o(X)$ is finite,}\\
\omega^{\omega^{o(X)+1}}& \text{if $o(X)= \varepsilon + n$, with $\varepsilon$ an epsilon number and $n< \omega$,}\\
\omega^{\omega^{o(X)}} & \text{otherwise.}
\end{array}
\right.\]
\end{theorem}

In section \ref{sec:an order-theoretic approach of the Howard-Bachmann ordinal number}, we will prove that a specific ordering $\T(\mathcal{B}(\cdot))$ is a well-partial-ordering with order type equal to the Howard-Bachman number. This ordering can be seen as a subordering of Friedman's famous ordering on finite trees and it uses the following constructor on $\wpo$'s.

\begin{definition}
Let $X$ be a partial order. Define $\mathcal{B}(X)$ as the partial order where the underlying set is the set of the finite structured binary trees with leaf-labels in $X$ and with the usual embeddability relation between trees. The internal nodes do not have labels. This means that if a tree $B \in \mathcal{B}(X)$ is embeddable in a tree $B' \in \mathcal{B}(X)$, then either $B$ and $B'$ are trees of one node with label $x \in X$ and $x' \in X$ respectively with $x \leq_X x'$ or $B'$ is a binary tree with a left immediate subtree $B'_1$ and a right immediate subtree $B'_2$ and $B$ is embeddable in $B'_1$ or $B'_2$ or $B$ has also a left immediate subtree $B_1$ and a right immediate subtree $B_2$ and $B_i$ is embeddable in $B'_i$ for $i=1,2$. This implies that if $B$ is embeddable in $B'$, then the internal nodes of $B$ are mapped on the internal nodes of $B'$ and the leaf-nodes of $B$ on the leaf-nodes of $B'$.
\end{definition}

One can prove that this construct a well-partial-ordering starting from a $\wpo$ $X$. We are interested in its maximal order type.

\begin{definition}
Define $\varphi_0 \beta$ as $\omega^\beta$ and let $\varphi_\alpha $ be the enumeration function of the common fixed points of all $\varphi_\gamma$ with $\gamma < \alpha$. This constructs the so-called Veblen hierarchy. 
\end{definition}

\begin{definition}
Let $\alpha$ be an ordinal.
\[\overline{\alpha} := \left\{ \begin{array}{ll}
\alpha-1 &  \text{if $\alpha$ is finite,}\\
\alpha+1 & \text{if $\alpha= \varphi_2 \beta +n$ with $n$ a natural number,}\\
\alpha & \text{otherwise.}
\end{array}
 \right. \]
\end{definition}

For a proof of the following theorem, we refer the reader to \cite{dianaschmidt}.
\begin{theorem}
If $X$ is a $\wpo$, then $\mathcal{B}(X)$ is a $\wpo$ and $o(\mathcal{B}(X)) = \varepsilon_{\overline{o(X)}}$.
\end{theorem}


\subsection{Generalized tree-structures}\label{sec:generalized tree-structures}

In section \ref{sec:an order-theoretic approach of the Howard-Bachmann ordinal number}, we present a $\wpo$ of order type $\eta_0$. This will be a tree-structure that can be interpreted as a subordering of Friedman's famous $\wpo$ on trees with gap-condition. In section \ref{sec:Tree-structures below the Howard-Bachmann ordinal number}, we give some auxiliary $\wpo$'s with maximal order types strictly below $\eta_0$. These $\wpo$'s can also be seen as tree-representations of ordinals below $\eta_0$. In this subsection, we give the definitions of these $\wpo$'s. The reader can also find this kind of $\wpo$'s in \cite{WellpartialorderingsandthebigVeblennumber} and \cite{computationmaximalordertypemultisets}.
For the actual proofs of well-partial-orderedness and the maximal order types, the reader has to wait until sections \ref{sec:Tree-structures below the Howard-Bachmann ordinal number} and \ref{sec:an order-theoretic approach of the Howard-Bachmann ordinal number}. Before we present the definition, we elaborate on the definition of theta-functions.


\begin{definition} Let $\Omega$ denote the first uncountable ordinal. Every ordinal $0<\alpha < \varepsilon_{\Omega+1}$ can be written as $\Omega^{\alpha_1} \beta_1 + \dots + \Omega^{\alpha_n}\beta_n$ with $\beta_i< \Omega$ and $\alpha> \alpha_1 > \dots > \alpha_n$. Define the set of coefficients recursively as $K(\alpha)=\{\beta_1,\dots,\beta_n\} \cup K(\alpha_1)\cup \dots \cup K(\alpha_n)$. Let $K(0)$ be $\{0\}$. Define then $k(\alpha)$ as the ordinal $\max( K(\alpha))$.
\end{definition}

\begin{definition}
For an ordinal $\alpha$, define $\Omega_0[\alpha]$ as $\alpha$ and $\Omega_{n+1}[\alpha]$ as $\Omega^{\Omega_n[\alpha]}$.
\end{definition}

\begin{definition} Let $P$ denote the set of the additive closed ordinal numbers $\{\omega^\alpha: \alpha \in ON\}$. For every ordinal $\alpha< \varepsilon_{\Omega+1}$, define $\vartheta(\alpha)$ as $\min\{\zeta \in P : k(\alpha) < \zeta \text{ and } \forall \beta< \alpha (k(\beta) < \zeta \rightarrow \vartheta(\beta) < \zeta)\}$. The Howard-Bachmann ordinal number is defined as $\eta_0 = \vartheta(\varepsilon_{\Omega+1}) = \sup_{n} \left(\vartheta\left(\Omega_n[1]\right)\right)$.
\end{definition}

For more information about the theta-function and its connection with Buchholz' $\Psi$-function, we refer the reader to \cite{rathjenweiermann}. There, they introduced the $\vartheta$-function in a different way, but one can prove that they coincide with our definition if the argument is above $\Omega^2$. It can be shown be an easy cardinality argument that $\vartheta\alpha < \Omega$.

\begin{lemma}\label{main property theta-function}
$\vartheta\alpha<\vartheta \beta \iff \begin{cases} 
\alpha<\beta\mbox{ and }k(\alpha) <\vartheta \beta \\
\beta<\alpha \mbox{ and } \vartheta \alpha\leq k(\beta). \end{cases}$
\end{lemma}

We need the following two additional lemmas. The proofs are rather straightforward.

\begin{lemma}\label{coefficients under multiplicative sum and product} Suppose $\alpha$ and $\beta$ are ordinals beneath $\varepsilon_{\Omega+1}$. Then 
\begin{align*}
k(\alpha \oplus \beta) &\leq k(\alpha) \oplus k(\beta),\\
k(\alpha \otimes \beta) &\leq \max\{k(\alpha)\oplus k(\beta), k(\alpha) \otimes k(\beta) \otimes \omega\},\\
k(\omega^{\alpha}) & \leq \omega^{k(\alpha)}.
\end{align*}
Furthermore, $k(\alpha), k(\beta) \leq k(\alpha \oplus \beta)$ and $k(\alpha) \leq k(\alpha \otimes \beta)$ if $\beta >0$.
\end{lemma}

\begin{lemma}\label{coefficients under *} Suppose $\alpha_n,\dots, \alpha_0$ are countable ordinal numbers with $\alpha_i < \gamma$ for an epsilon number $\gamma$. Then  $k(o((\Omega^n \alpha_n + \dots + \Omega \alpha_1 + \alpha_0)^*))< \gamma$.
\end{lemma}

Before we give the definition of the $\wpo$'s that we use in sections \ref{sec:Tree-structures below the Howard-Bachmann ordinal number} and \ref{sec:an order-theoretic approach of the Howard-Bachmann ordinal number}, let us define a specific class of constructors.

\begin{definition}
Define $Map$ as the least set satisfying the following:
\begin{enumerate}
\item $\cdot  \in Map$, ($\cdot$ plays the role of a place holder).
\item If $\mathbb{X}$ is a countable $\wpo$, then $\mathbb{X} \in Map$,
\item If $W_1, W_2 \in Map$, then $W_1 + W_2$, $W_1 \times W_2$, $W_1^*$  and $\mathcal{B}(W_1)$ are also elements of $Map$. 
\end{enumerate} 
Every element $W$ of $Map$ can be seen as a mapping from the set of partial orderings to the set of partial orderings: $W(X)$ is a partial ordering by putting the partial order $X$ into the $\cdot$. For example, if $W=  (\mathcal{B}(\cdot)\times \X)^* $, then $W(X)$ is the partial ordering $ (\mathcal{B}(X)\times \X)^* $. Furthermore, if $X$ is a $\wpo$, then $W(X)$ is a $\wpo$ and if $X$ is countable, then so is $W(X)$.
Every element of $W(X)$ is represented by a term in finitely many elements in $X$. For example in the case 
$W= (\mathcal{B}(\cdot)\times \X)^* $, the element $\left((B(x_1,x_2),\mathbb{x}_1),(B(x_3,x_1),\mathbb{x}_2)\right)$ in $W(X)$ with $x_1,x_2,x_3 \in X$, $\mathbb{x}_1, \mathbb{x}_2 \in \X$ and  $B(a,b)$ the binary tree\\
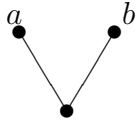
\begin{picture}(1,56)
\put(60,10){\circle*{5}}
\put(42,40){\circle*{5}}
\put(78,40){\circle*{5}}
\put(60,10){\line(3,5){18}}
\put(60,10){\line(-3,5){18}}
\put(37,43){$a$}
\put(81,43){$b$}
\end{picture}
\ \\
is represented by a term in $x_1,x_2,x_3$. By deleting all entries of $X$, we get the naked term $\left((B(\cdot,\cdot),\mathbb{x}_1),(B(\cdot,\cdot),\mathbb{x}_2)\right)$, which we notate as $w(\cdot,\cdot,\cdot,\cdot)$, where
\[ w(a,b,c,d) = \left((B(a,b),\mathbb{x}_1),(B(c,d),\mathbb{x}_2)\right).\]
Therefore, the element $\left((B(x_1,x_2),\mathbb{x}_1),(B(x_3,x_1),\mathbb{x}_2)\right)$ can be described using this naked term $w(\cdot,\cdot,\cdot,\cdot)$ and the elements $x_1,x_2,x_3\in X$ as $w(x_1,x_2,x_3,x_1)$. In general, an element of $W(X)$ is represented as $w(x_1,\dots,x_n)$ using a naked term $w(\cdot,\dots,\cdot)$ and elements $x_1,\dots,x_n\in X$. We will call this naked term `an element of $W$'.
\end{definition}

\begin{definition}
Take $W \in Map$. Define $\T(W)$ as the least set satisfying the following requirements
\begin{enumerate}
\item $\circ \in \T(W)$,
\item If $w(\cdot ,\dots,\cdot)$ is an element of $W$ and $t_1,\dots,t_n \in \T(W)$, then $w(t_1,\dots,t_n)$ is an element of $W(\T(W))$ and let $\circ[w(t_1,\dots,t_n)] \in \T(W)$. We will say that $t_1,\dots,t_n$ have a lower complexity than $\circ[w(t_1,\dots,t_n)]$.
\end{enumerate}
Let the underlying ordering $\leq_{\T(W)}$  be the least binary reflexive and transitive relation on $\T(W)$ such that
\begin{enumerate}
\item $\circ \leq_{\T(W)}  t$  for every $t$ in $\T(W)$,
\item if $s  \leq_{\T(W)}  t_j$ for a certain $j$, then $s \leq_{\T(W)} \circ[w(t_1,\dots,t_n)]$,
\item if $w(t_1,\dots,t_n)\leq_{W(\T(W), \leq_{\T(W)})} w'(t'_1,\dots,t'_{n'})$,\\
then $\circ[w(t_1,\dots,t_n)] \leq_{\T(W)} \circ[w'(t'_1,\dots,t'_{n'})]$.
\end{enumerate}
If it is clear from the context, we also notate $\leq_{\T(W)}$ as $\leq$. Sometimes, we notate the elements $\circ[w(t_1,\dots,t_n)]$ also as $\circ w(t_1,\dots,t_n)$ if $w(t_1,\dots,t_n)$ has already enough brackets in its description.
\end{definition}

\begin{notation} Suppose $t=\circ[w(t_1,\dots,t_n)]$ is an element of $\T(W)$. We denote the element $w(t_1,\dots,t_n)$ of $W(\T(W))$ also as $\times t$.
\end{notation}

Our general conjecture is that for every $W \in Map$, the partial ordering $\T(W)$ is actually a $\wpo$ and the maximal order type is equal to $\vartheta(o(W(\Omega)))$ if $\Omega^3 \leq o(W(\Omega)) \leq \varepsilon_{\Omega+1}$. In section \ref{sec:Tree-structures below the Howard-Bachmann ordinal number}, we prove for specific $W \in Map$ that the ordering $\T(W)$ is indeed a $\wpo$ and $\vartheta(o(W(\Omega)))$ is an upper bound on the maximal order type of $\T(W)$. In section \ref{sec:an order-theoretic approach of the Howard-Bachmann ordinal number}, we show that $\T(\mathcal{B}(\cdot))$ is also a $\wpo$ and $\vartheta(\varepsilon_{\Omega+1})$ is exactly equal to $o(\T(\mathcal{B}(\cdot)))$. This $\wpo$ can be seen as a subordering of Friedman's famous trees with gap-embeddability relation \cite{simpsonfinitetrees}. 
For cases where $o(W(\Omega)) > \varepsilon_{\Omega+1}$, some little adaptations of the general formula $o(\T(W)) = \vartheta(o(W(\Omega)))$ are needed because the domain of the theta-function is below $\varepsilon_{\Omega+1}$. This is however beyond the scope of this article and will be treated in latter work. We believe that generalizations will lead to a full classification of the strength of Friedman's $\wpo$'s.

\medskip

The next lemma is a very important lemma for the rest of the article. We will skip its proof, but one can find some subparts and the general idea of this proof in \cite{WellpartialorderingsandthebigVeblennumber}.

\begin{lemma}[Lifting Lemma]\label{Lifting} Assume that $W \in Map$ and let $q$ be a quasi-embedding from the partial ordering $Y$ to the partial ordering $Z$. Then for all elements $y_1,\dots, y_n, y'_1, \dots,y'_m$ in $Y$ and $w(\cdot,\dots,\cdot),$ $v(\cdot,\dots,\cdot)$ in $W$ the inequa\-lity $w(q(y_1),\dots, q(y_n)) \leq_{W(Z)} v(q(y'_1),\dots,q(y'_m)) $ implies $w(y_1,\dots, y_n) \leq_{W(Y)} v(y'_1,\dots,y'_m)$.
\end{lemma}

Before we go further, we want to show that $\T(\mathcal{B}(\cdot))$ can indeed be seen as a subordering of Friedman's $\wpo$ with gap-condition. First, we give the definition of his $\wpo$.

\begin{definition}
Let $\mathbb{T}_n$ be the set of finite rooted trees with labels in $\{0,\dots,n-1\}$. An element of $\mathbb{T}_n$ is of the form $(T,l)$, where $T$ is a finite rooted tree, which we see as a partial ordering on a set of nodes, and $l$ is a labeling function, a mapping from $T$ to the set $\{0,\dots,n-1\}$. Define $(T_1,l_1)\leq_{gap} (T_2,l_2)$ if there exists an injective order- and infimum-preserving mapping $f$ from $T_1$ to $T_2$ such that 
\begin{enumerate}
\item $\forall \tau \in T_1$, we have $l_1(\tau) = l_2(f(\tau))$.
\item $\forall \tau  \in T_1$ and for all immediate successors $\tau' \in T_1$ of $\tau$, we have
that if $\overline{\tau} \in T_2$ and $f(\tau) < \overline{\tau} < f(\tau')$, then $l_2(\overline{\tau}) \geq l_2(f(\tau')) = l_1(\tau')$.
\end{enumerate}
\end{definition}

In \cite{simpsonfinitetrees}, this is the so-called \textit{weak} gap-embeddability relation.

\begin{definition}
Define the partial ordering $\overline{\mathbb{T}}_2$ as the subset of $(\mathbb{T}_2, \leq_{gap})$ which consists of all finite rooted trees such that nodes with label $0$ has zero or one immediate successor(s) and nodes with label $1$ has exactly two immediate successors. Furthermore, every tree in $\overline{\mathbb{T}}_2$ has a root with label $0$.
\end{definition}

Note that the trees in $\mathbb{T}_n$ are unstructured. The trees in $\T(\mathcal{B}(\cdot))$ are structured, so to see the resemblance between $\T(\mathcal{B}(\cdot))$ and $\overline{\mathbb{T}}_2$ we have to restrict $\overline{\mathbb{T}}_2$ even a little bit more: we say that every tree in $\overline{\mathbb{T}}_2$ is structured, meaning that it has a left-hand side and a right-hand side and an embedding between two trees preserves these left-right order.

\begin{lemma}\label{T(B) is a subordering of friedmans wpo}
The partial-ordering $\T(\mathcal{B}(\cdot))$ is order-isomorphic to the partial ordering $\overline{\mathbb{T}}_2$.
\end{lemma}
\begin{proof}
Define $g:\T(\mathcal{B}(\cdot)) \to \overline{\mathbb{T}}_2$ as follows. Let $g(\circ)$ be the tree which consists of one node with label $0$. Take $t= \circ[B(t_1,\dots,t_n)]$ with $B(t_1,\dots,t_n)$ a binary tree with leaf-labels in the set $\{t_1,\dots,t_n\}$ and assume that $g(t_1),\dots,g(t_n)$ is already defined. Set $g(t)$ then as the tree consisting of a root with label $0$, that root connected with an edge to the root of $B(t_1,\dots,t_n)$. Give all the internal nodes of $B$ label $1$ and plug $g(t_i)$ in the leaves of $B(t_1,\dots,t_n)$ with label $t_i$ for every $i$.
For example, if $t = \circ[B(\circ,\circ[B(\circ,\circ)])]$, with $B(a,b)$ equal to\\
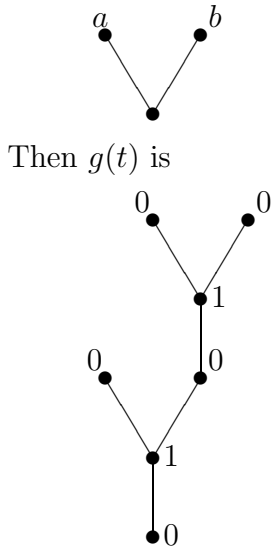
\begin{picture}(1,66)
\put(60,20){\circle*{5}}
\put(42,50){\circle*{5}}
\put(78,50){\circle*{5}}
\put(60,20){\line(3,5){18}}
\put(60,20){\line(-3,5){18}}
\put(37,53){$a$}
\put(81,53){$b$}
\end{picture}
Then $g(t)$ is\\
\begin{picture}(1,156)
\put(60,20){\circle*{5}}
\put(60,20){\line(0,1){30}}
\put(60,50){\circle*{5}}
\put(42,80){\circle*{5}}
\put(78,80){\circle*{5}}
\put(60,50){\line(3,5){18}}
\put(60,50){\line(-3,5){18}}

\put(78,80){\line(0,1){30}}
\put(78,110){\circle*{5}}
\put(60,140){\circle*{5}}
\put(96,140){\circle*{5}}
\put(78,110){\line(3,5){18}}
\put(78,110){\line(-3,5){18}}

\put(35,83){$0$}
\put(81,83){$0$}
\put(53,143){$0$}
\put(99,143){$0$}
\put(64,17){$0$}

\put(82,107){$1$}
\put(64,47){$1$}
\end{picture}
It is easy to see that $g$ is surjective. If we can prove that 
\[t\leq_{\T(\mathcal{B}(\cdot))} t' \Leftrightarrow g(t) \leq_{\overline{\mathbb{T}}_2} g(t'),\] 
we are done. We will prove this by induction on the sum of complexities of $t$ and $t'$. If $t= \circ$ or $t' = \circ$, then this is trivial. Assume both $t$ and $t'$ are different from $\circ$. Let $t = \circ[B(t_1,\dots,t_n)]$ and $t' = \circ[B'(t'_1,\dots,t'_m)]$. If $t \leq_{\T(\mathcal{B}(\cdot))} t'$, then either $t \leq_{\T(\mathcal{B}(\cdot))} t'_i$ for a certain $i$ or $B(t_1,\dots,t_n) \leq_{\mathcal{B}({\T(\mathcal{B}(\cdot))})} B'(t'_1,\dots,t'_m)$. In both cases, the induction hypothesis yields $g(t) \leq_{\overline{\mathbb{T}}_2} g(t')$ quite easily.
Now assume $g(t) \leq_{\overline{\mathbb{T}}_2} g(t')$. We know that the root of $g(t)$, which has label $0$, is mapped on a node with label $0$. If it is not mapped onto the root of $g(t')$, then it is mapped onto a node with label $0$ in $g(t'_i)$ for a certain $i$. Hence $g(t) \leq_{\overline{\mathbb{T}}_2} g(t'_i)$, so $t\leq_{\T(\mathcal{B}(\cdot))} t'_i \leq_{\T(\mathcal{B}(\cdot))} t'$. Now assume that the root of $g(t)$ is mapped onto the root of $g(t')$. Every internal node $a$ of $B$ has to be mapped on an internal node of $B'$, because otherwise the label $0$ of the root of the $g(t_i)$ in which the internal node $a$ of $B$ is mapped, gives a contradiction with the gap-condition. Furthermore, every leaf of $B$, in which $g(t_i)$ are plugged in, has label $0$ and is mapped on a node in $g(t')$ with label $0$. We can conclude that $B(g(t_1),\dots,g(t_n)) \leq_{\mathcal{B}(\overline{\mathbb{T}}_2)} B'(g(t'_1),\dots,g(t'_m))$. The induction hypothesis yields that $g$ is a quasi-embedding from the set $\{t_1,\dots,t_n,t'_1,\dots,t'_m\}$ to $\overline{\mathbb{T}}_2$. So the Lifting Lemma implies $B(t_1,\dots,t_n) \leq_{\mathcal{B}({\T(\mathcal{B}(\cdot))})} B'(t'_1,\dots,t'_m)$. Hence,  $t\leq_{\T(\mathcal{B}(\cdot))}  t'$.

\end{proof}

From the previous lemma, one can actually already conclude that $\T(\mathcal{B}(\cdot))$ is a $\wpo$. Therefore, one can think that the well-partial-orderedness proof of $\T(\mathcal{B}(\cdot))$ in Theorem \ref{computation upper bound T(W) with W(X)=B(X)} is superfluous. However, this well-partial-orderedness proof does not need an extra argument: it follows from the calculation of an upper bound of the maximal order type of $\T(\mathcal{B}(\cdot))$. Therefore, we do not really waste efforts by stating it in Theorem \ref{computation upper bound T(W) with W(X)=B(X)}.


\section{Tree-structures below the Howard-Bachmann ordinal}\label{sec:Tree-structures below the Howard-Bachmann ordinal number}

In section \ref{sec:an order-theoretic approach of the Howard-Bachmann ordinal number}, we will show that $\T(\mathcal{B}(\cdot))$ is a $\wpo$ with maximal order type $\vartheta(\varepsilon_{\Omega+1})$. For obtaining these results, we need to approximate this $\wpo$. This is done from `below' and is treated in this section. The next theorems are generalizations of Theorems 9 and 10 in \cite{WellpartialorderingsandthebigVeblennumber}. The proofs follow the same procedures as in that article, but they are more involved. 

\begin{theorem}\label{Computation upper bound for T(sum(sum X^* X))2}  Suppose $\Y_{i,j,k}$ and $\Z_{i}$ are countable wpo's for all indices. If 
\begin{align*}
W(X) = \sum_{i=0}^N \left(\left(\sum_{j=0}^{k_{i,1}} \Y_{i,j,1} \times X^{j}\right)^* \times \dots \times \left(\sum_{j=0}^{k_{i,{n_i}}} \Y_{i,j,{n_i}} \times X^{j}\right)^*  \times X^{m_i} \times \Z_i \right)
\end{align*}
then $\T(W)$ is a $\wpo$ and $o(\T(W))\leq \vartheta (o(W(\Omega)))$.
\end{theorem}
\begin{proof} We will prove the theorem by main induction on the ordinal $o(W(\Omega))$. 
If $W(X)$ is the empty $\wpo$ for every $X$, then the theorem follows easily. We can now assume without loss of generality that $\Z_i$, $\Y_{i,k_{i,j},j}$ are non-empty and $k_{i,j} >0$ for all $i$ and $l$.
If $o(W(\Omega)) < \Omega$, then $\cdot$ does not occur in $W$. Therefore, $n_i = m_i=0$ for all $i$. Hence, $W(X)$ is equal to a $\wpo$ $\sum_{i=0}^N \Z_i =: \Z$. So $\T(W) \cong \Z \cup \{0\}$, where $0$ is a new element smaller than every element in $\Z$. This yields that $\T(W)$ is a $\wpo$ and $o(\T(W)) \leq o(\Z) +1 \leq \vartheta(o(\Z)) =  \vartheta(o(W(\Omega)))$.\\
Assume from now on that $o(W(\Omega)) \geq \Omega$. 
We will prove that $L(t)$ is a $\wpo$ and $l(t) < \vartheta(o(W(\Omega)))$ for every $t$ in $\T(W)$. Then the theorem follows from Lemma \ref{dejonghandparikh2}. If $t=\circ$, then $L(t)$ is the empty $\wpo$ and $l(t)=0<\vartheta(o(W(\Omega)))$. Assume now 
\[t=\circ((\overline{t_1},\dots,\overline{t_{n_l}}),(t_1,\dots,t_{m_l}),z)\] 
with $z \in \Z_{l}$, $t_p \in \T(W)$ and $\overline{t_p}$ an element in $\left(\sum_{j=0}^{k_{l,p}} \Y_{l,j,p} \times \T(W)^{j}\right)^*$, meaning
\begin{align*}
\overline{t_p}= (\overline{\overline{t^p_1}},\dots, \overline{\overline{t^p_{r_p}}}),
\end{align*}
with $\overline{\overline{t^p_q}}\in \sum_{j=0}^{k_{l,p}} \Y_{l,j,p} \times \T(W)^{j}$. So
\begin{align*}
\overline{\overline{t^p_q}} = \left(y_{p,q} , \left(t^{p,q}_1,\dots, t^{p,q}_{v_{p,q}}\right)\right)
\end{align*}
with $y_{p,q} \in \Y_{l,v_{p,q},p}$ and $v_{p,q} \leq k_{l,p}$. Assume $l(t_i), l(t^{i,j}_{k})<\vartheta(o(W(\Omega)))$ and $L(t_i)$ and $L(t^{i,j}_k)$ are $\wpo$'s. We want to prove that $L(t)$ is a $\wpo$ and $l(t)<\vartheta(o(W(\Omega)))$.

\medskip

Suppose $s$ is an arbitrary element in $\T(W)$, different from $\circ$. Then
\begin{align}
\nonumber s&=\circ((\overline{s_1},\dots,\overline{s_{n_{l'}}}),(s_1,\dots,s_{m_{l'}}),z'),\\
\label{definition s in sumX^*}\overline{s_p} &= (\overline{\overline{s^p_1}},\dots,\overline{\overline{s^p_{r'_p}}}),\\
\nonumber \overline{\overline{s^p_q}} &= \left(y'_{p,q},(s^{p,q}_1,\dots, s^{p,q}_{w_{p,q}})\right),
\end{align}
with
\begin{align*}
z' &\in  \Z_{l'},\\
y'_{p,q} & \in  \Y_{l',w_{p,q},p}.
\end{align*}

\medskip 

We see that $s\in L(t)$ iff $s_i \in L(t)$, $s^{i,j}_k \in L(t)$ and one of the following holds:
\begin{enumerate}
\item[a.] $l \neq l'$,
\item[b.] $l = l'$, $z' \in L_{\Z_l}(z)$,
\item[c.]\label{nummer 1} $l = l'$, $z \leq_{\Z_l} z'$, $(t_1,\dots,t_{m_l}) \not\leq (s_1,\dots,s_{m_{l'}})$,
\item[d.]\label{nummer 2} $l = l'$, $z \leq_{\Z_l} z'$, $(t_1,\dots,t_{m_l}) \leq (s_1,\dots,s_{m_{l'}})$, $(\overline{t_1},\dots,\overline{t_{n_l}}) \not\leq (\overline{s_1},\dots,\overline{s_{n_{l'}}})$,
\end{enumerate}
Now if (c.) holds, there must be a minimal index $m(s) \leq m_l$ such that
\begin{align*}
t_1\leq s_1 ,\dots, t_{m(s)-1}\leq s_{m(s)-1}, t_{m(s)} \not\leq s_{m(s)}.
\end{align*}
If (d.) is valid, there must be a minimal index $n(s) \leq n_l$ such that
\begin{align*}
\overline{t_1}\leq \overline{s_1} ,\dots, \overline{t_{n(s)-1}}\leq \overline{s_{n(s)-1}}, \overline{t_{n(s)}} \not\leq \overline{s_{n(s)}}.
\end{align*}
By similar arguments,
\begin{align*}
\overline{t_{n(s)}} =(\overline{\overline{t^{n(s)}_1}},\dots,\overline{\overline{t^{n(s)}_{r_{n(s)}}}})\not\leq (\overline{\overline{s^{n(s)}_1}},\dots,\overline{\overline{s^{n(s)}_{r'_{n(s)}}}})=\overline{s_{n(s)}}
\end{align*}
holds iff we are in one of the next cases
\begin{enumerate}[\ \ \ a.]
\item[$1$.] $\overline{\overline{t^{n(s)}_1}} \not\leq \overline{\overline{s^{n(s)}_i}}$ for every $i$,
\item[$2$.] there exists an index $i_1$ such that $\overline{\overline{t^{n(s)}_1}} \not\leq \overline{\overline{s^{n(s)}_i}}$ for every $i<i_1$, $\overline{\overline{t^{n(s)}_1}} \leq \overline{\overline{s^{n(s)}_{i_1}}}$ and $\overline{\overline{t^{n(s)}_2}} \not\leq \overline{\overline{s^{n(s)}_i}}$ for every $i>i_1$,
\item[] \dots
\item[$r_{n(s)}$.] there exist indices $i_1 < \dots < i_{r_{n(s)}-1}$ such that $\overline{\overline{t^{n(s)}_1}} \not\leq \overline{\overline{s^{n(s)}_i}}$ for every $i<i_1$, $\overline{\overline{t^{n(s)}_1}} \leq \overline{\overline{s^{n(s)}_{i_1}}}$ and $\overline{\overline{t^{n(s)}_2}} \not\leq \overline{\overline{s^{n(s)}_i}}$ for every $i_2>i>i_1$,\dots,$\overline{\overline{t^{n(s)}_{r_{n(s)}-1}}} \leq \overline{\overline{s^{n(s)}_{i_{r_{n(s)}-1}}}}$ and $\overline{\overline{t^{n(s)}_{r_{n(s)}}}} \not\leq \overline{\overline{s^{n(s)}_i}}$ for every $i>i_{r_{n(s)}-1}$.
\end{enumerate}
Now,
\begin{align*}
\overline{\overline{t^{n(s)}_{i}}}= \left(y_{n(s),i} , \left(t^{n(s),i}_1,\dots, t^{n(s),i}_{v_{n(s),i}}\right)\right)  \not\leq \left(y'_{n(s),j},(s^{n(s),j}_1,\dots ,s^{n(s),j}_{w_{n(s),j}})\right) =\overline{\overline{s^{n(s)}_{j}}}
\end{align*}
holds iff one of the following is valid
\begin{enumerate}
\item $v_{n(s),i} \neq w_{n(s),j}$,
\item $v_{n(s),i} = w_{n(s),j}$ and $y_{n(s),i} \not\leq_{\Y_{l, v_{n(s),i},n(s)}} y'_{n(s),j}$,
\item $v_{n(s),i} = w_{n(s),j}$, $y_{n(s),i} \leq_{\Y_{l,v_{n(s),i},n(s)}} y'_{n(s),j}$ and there exists a minimal index $p_{i,j}(s) \leq v_{n(s),i}$ such that 
\[
t^{n(s),i}_1 \leq s^{n(s),j}_1,\dots, t^{n(s),i}_{p_{i,j}(s)-1} \leq s^{n(s),j}_{p_{i,j}(s)-1},
t^{n(s),i}_{p_{i,j}(s)} \not\leq s^{n(s),j}_{p_{i,j}(s)}.
\]
\end{enumerate}
We just completely characterized $L(t)$. Using this characterization, we define the following constructor $W'(X)$ in $Map$: let $W'(X)$ be $W'_1(X)+ W'_2(X)$ with $W'_1(X)$ equal to
\begin{align*}
 &\sum_{i=0, i \neq l}^N \left(\left(\sum_{j=0}^{k_{i,1}} \Y_{i,j,1} \times X^{j}\right)^* \times \dots \times \left(\sum_{j=0}^{k_{i,{n_i}}} \Y_{i,j,{n_i}} \times X^{j}\right)^*  \times X^{m_i} \times \Z_i \right)\\
+& \left(\left(\sum_{j=0}^{k_{l,1}} \Y_{l,j,1} \times X^{j}\right)^* \times \dots \times \left(\sum_{j=0}^{k_{l,{n_l}}} \Y_{l,j,{n_l}} \times X^{j}\right)^*  \times X^{m_l} \times L_{\Z_l}(z) \right)\\
+& \sum_{m=1}^{m_l} \left(\left(\sum_{j=0}^{k_{l,1}} \Y_{l,j,1} \times X^{j}\right)^* \times \dots \times \left(\sum_{j=0}^{k_{l,{n_l}}} \Y_{l,j,{n_l}} \times X^{j}\right)^* \right. \\
& \ \ \ \ \ \ \ \  \ \ \left. {}^{{}^{{}^{{}^{{}^{{}^{{}^{{}^{{}^{{}}}}}}}}}}   \times \,  X^{m_l-1} \times L_{\T(W)}(t_m) \times \Z_l \right)
\end{align*}
and $W'_2(X)$ equal to
\begin{align*}
&\sum_{n=1}^{n_l} \sum_{q=1}^{r_n} \left[\left(\sum_{j=0}^{k_{l,1}} \Y_{l,j,1} \times X^{j}\right)^* \times \dots \times \left(\sum_{j=0}^{k_{l,n-1}} \Y_{l,j,n-1} \times X^{j}\right)^*  \right.\\ 
 &\ \ \ \ \ \ \ \ \ \ \ \ \ \ \ \  \times \left(\sum_{j=0}^{k_{l,n+1}} \Y_{l,j,n+1} \times X^{j}\right)^* \times \dots  \times \left(\sum_{j=0}^{k_{l,{n_l}}} \Y_{l,j,{n_l}} \times X^{j}\right)^*  \times X^{m_l} \times \Z_l \\ 
&\ \ \ \ \ \ \ \ \ \ \ \ \ \  \ \ \times \left. \left(\sum_{j=0}^{k_{l,n}}  \Y_{l,j,n} \times X^{j}\right)^{q-1} \times (W'^1_2(X))^* \times \dots \times (W'^q_2(X))^* \right],\\
\end{align*}
with 
\begin{align*}
W'^i_2(X) &:=
\sum_{j=0,j\neq v_{n,i}}^{k_{l,n}} \left( \Y_{l,j,n} \times X^{j}\right)
\\
&  \ \ \ \ \ \ \ \ + \left( L_{\Y_{l,v_{n,i},n}}(y_{n,i}) \times X^{v_{n,i}}\right)  + \sum_{p=1}^{v_{n,i}}   \left( \Y_{l,v_{n,i},n} \times X^{v_{n,i}-1} \times L_{\T(W)}(t^{n,i}_{p})  \right). 
\end{align*}
The three cases separated by the sign $+$ in $W'_1(X)$ corresponds to the cases (a.), (b.) and (c.). The index $m$ in the third line of $W'_1(X)$ matches with $m(s)$. $W'_2(X)$ corresponds to case (d.). The index $n$ in $W'_2(X)$ matches with $n(s)$ and the index $q$ corresponds to the cases 1.-\dots-$r_{n(s)}$ in which we are for $\overline{t_{n(s)}}\not\leq \overline{s_{n(s)}}$. $W'^i_2(X)$ matches with 
$\overline{\overline{t^{n(s)}_{i}}} \not\leq \overline{\overline{s^{n(s)}_{j}}}$, where $p$ in $W'^i_2(X)$ corresponds to $p_{i,j}(s)$.

\medskip
Now, following the here-described characterization of $L(t)$ thoroughly step-by-step, one can see that there exists a mapping $f$ from $\{w(s_1,\dots,s_n) \in W(\T(W)): \circ[w(s_1,\dots,s_n)] \in L(t) \}$ into $W'(\T(W))$ such that if we have $f(w(s_1,\dots,s_n)) = w'(s'_1,\dots,s'_m)$, then $\{s'_1,\dots,$ $s'_m\} \subseteq \{s_1,\dots,s_n\} \subseteq L(t)$ and if the inequality $f(w(s_1,\dots,s_n))  \leq_{W'(\T(W))} f(\overline{w}(\overline{s_1},\dots,\overline{s_k}))$ holds, then $w(s_1,\dots,s_n)  \leq_{W(\T(W))} \overline{w}(\overline{s_1},\dots,\overline{s_k})$. We do not explicitly write out the full details of this argument because in Theorem \ref{computation upper bound T(W) with W(X)=B(X)} we will do a similar proof (written out in full details) and that proof is less messy.

\medskip

We pinpoint a mapping $g$ from $L(t)$ into $\T(W')$. This mapping will be a quasi-embedding. We do this by induction on the complexity of the terms in $L(t)$. Let $g(\circ)$ be $\circ$. Let $s$ be an element of $L(t)$, defined as in (\ref{definition s in sumX^*}). By induction, we can assume that $g(s_i)$ and $g(s^{i,j}_k)$ are already defined. We know that $s$ is equal to 
\[\circ\left[w(s_1,\dots,s_{m_{l'}},s^{1,1}_1,\dots, s^{n_{l'},r'_{n_{l'}}}_{w_{n_{l'},r'_{n_{l'}}}})\right]\] 
for a certain element $w(\cdot,\dots,\cdot)$ in $W$. Denote 
\[f(\times s) = f\left(w(s_1,\dots,s_{m_{l'}},s^{1,1}_1,\dots, s^{n_{l'},r'_{n_{l'}}}_{w_{n_{l'},r'_{n_{l'}}}})\right)\]
as $w'(s'_1,\dots,s'_m) \in W'(L(t))$ with $\{s'_1,\dots,s'_m\} \subseteq \left\{s_1,\dots,s_{m_{l'}},s^{1,1}_1,\dots, s^{n_{l'},r'_{n_{l'}}}_{w_{n_{l'},r'_{n_{l'}}}}\right\}$. Define $g(s)$ as $\circ[w'(g(s'_1),\dots,g(s'_m))]$. We want to prove that the mapping $g$ is a quasi-embedding.

\medskip

We prove by induction on the sum of the complexities of $s$ and $\overline{s}$ that the inequality $g(s) \leq_{\T(W')} g(\overline{s})$ implies $s \leq_{\T(W)} \overline{s}$. If $s$ or $\overline{s}$ is equal to $\circ$, then this is trivial. Now, let $s = \circ[w(s_1,\dots,s_n)]$ and $\overline{s} = \circ[\overline{w}(\overline{s}_1,\dots,\overline{s}_{\overline{n}})]$. Assume that $f(w(s_1,\dots,s_n))$ is equal to $w'(s'_1,\dots,s'_m)$ and  $f(\overline{w}(\overline{s}_1,\dots,\overline{s}_{\overline{n}}))$ is $\overline{w'}(\overline{s}'_1,\dots,\overline{s}'_{\overline{m}})$. So we have 
\[\circ[w'(g(s'_1),\dots,g(s'_m))]= g(s) \leq_{\T(W')} g(\overline{s}) = \circ[\overline{w'}(g(\overline{s}'_{1}), \dots, g(\overline{s}'_{\overline{m}}))].\]
Then either $g(s)  \leq_{\T(W')} g(\overline{s}'_q)$ for a certain $q$ or
\[w'(g(s'_1),\dots,g(s'_m)) \leq_{W'(\T(W'))}  \overline{w'}(g(\overline{s}'_{1}), \dots, g(\overline{s}'_{\overline{m}})).\]
In the first case, $s \leq \overline{s}'_q \leq \overline{s}$. In the latter case, the induction hypothesis yields that $g$ is a quasi-embedding from the set $\{s'_1,\dots,s'_m,\overline{s}'_{1}, \dots, \overline{s}'_{\overline{m}}\}$ to $\T(W')$. Hence, the Lifting Lemma implies $w'(s_1,\dots,s'_m) \leq_{W'(\T(W))} \overline{w'}(\overline{s}'_1, \dots, \overline{s}'_{\overline{m}})$. So $f(\times s) \leq f(\times \overline{s})$, hence $\times s \leq \times \overline{s}$. From this we can conclude that $s \leq \overline{s}$.

\medskip

So $g$ is a quasi-embedding from $L(t)$ in $\T(W')$. If $o(W'(\Omega)) < o(W(\Omega))$, the main induction hypothesis yields $\T(W')$ is a $\wpo$ and $o(\T(W')) \leq \vartheta(o(W'(\Omega)))$. Therefore, using Lemma \ref{quasi-embedding}, $L(t)$ is a $\wpo$ and $o(L(t))\leq \vartheta(o(W'(\Omega)))$. 
If additio\-nally $k(o(W'(\Omega))) < \vartheta(o(W(\Omega)))$, we can conclude that the inequality $o(L(t)) < \vartheta(o(W(\Omega)))$ holds, the objective that we want to achieve. Thus if we can prove that $o(W'(\Omega)) < o(W(\Omega))$ and $k(o(W'(\Omega))) < \vartheta(o(W(\Omega)))$, we can end the proof of this theorem.

\medskip

\underline{\textbf{1) $o(W'(\Omega)) < o(W(\Omega))$.}}\\
For notational convenience, we write $\Y$ instead of $o(\Y)$ for $\wpo$'s $\Y$. Additionally, we write $\alpha^*$ instead of $o(\alpha^*)$ for ordinal numbers $\alpha$. 
\begin{align}
\nonumber o(W'^i_2(\Omega)) &< \bigoplus_{j=0,j\neq v_{n,i}}^{k_{l,n}} \left( \Y_{l,j,n} \otimes \Omega^{j}\right)
\oplus \left( l_{\Y_{l,v_{n,i},n}}(y_{n,i}) \otimes \Omega^{v_{n,i}}\right) \oplus \Omega^{v_{n,i}}\\
\label{strictinequalityW'^i_2} &\leq 
\bigoplus_{j=0}^{k_{l,n}} \left( \Y_{l,j,n} \otimes \Omega^{j}\right).
\end{align}
We know that $\left(\bigoplus_{j=0}^{k_{l,n}}  \Y_{l,j,n} \otimes \Omega^{j}\right)^{*}$ is a multiplicative closed ordinal. Therefore, inequality (\ref{strictinequalityW'^i_2}) yields
\begin{align*}
 \left(\bigoplus_{j=0}^{k_{l,n}}  \Y_{l,j,n} \otimes \Omega^{j}\right)^{q-1} \otimes o(W'^1_2(\Omega)^*) \otimes \dots \otimes o(W'^q_2(\Omega)^*)< \left(\bigoplus_{j=0}^{k_{l,n}}  \Y_{l,j,n} \otimes \Omega^{j}\right)^{*}.
\end{align*}
From the assumption that $k_{l,n} >0$ and $\Y_{l,k_{l,n},n} \neq \emptyset$, we also have
\begin{align*}
 &{}\Z_l \otimes \left(\bigoplus_{j=0}^{k_{l,n}}  \Y_{l,j,n} \otimes \Omega^{j}\right)^{q-1} \otimes o(W'^1_2(\Omega)^*) \otimes \dots \otimes o(W'^q_2(\Omega)^*)\\
 <{}&{} \left(\bigoplus_{j=0}^{k_{l,n}}  \Y_{l,j,n} \otimes \Omega^{j}\right)^{*}.
\end{align*}
Because 
\[\left(\bigoplus_{j=0}^{k_{l,1}} \Y_{l,j,1} \otimes \Omega^{j}\right)^* \otimes \dots \otimes \left(\bigoplus_{j=0}^{k_{l,{n_l}}} \Y_{l,j,{n_l}} \otimes \Omega^{j}\right)^*  \otimes \Omega^{m_l} \] 
is additive closed, we obtain
\begin{align*}
o(W'_2(X)) < \left(\bigoplus_{j=0}^{k_{l,1}} \Y_{l,j,1} \otimes \Omega^{j}\right)^* \otimes \dots \otimes \left(\bigoplus_{j=0}^{k_{l,{n_l}}} \Y_{l,j,{n_l}} \otimes \Omega^{j}\right)^*  \otimes \Omega^{m_l}
\end{align*}
and
\begin{align*}
 &\bigoplus_{m=1}^{m_l} \left( \left(\bigoplus_{j=0}^{k_{l,1}} \Y_{l,j,1} \otimes \Omega^{j}\right)^* \otimes \dots \otimes \left(\bigoplus_{j=0}^{k_{l,{n_l}}} \Y_{l,j,{n_l}} \otimes \Omega^{j}\right)^* \right. \\
& \ \ \ \ \ \ \ \ \ \ \left.  {}^{{}^{{}^{{}^{{}^{{}^{{}^{{}^{{}^{{}}}}}}}}}} \otimes \Omega^{m_l-1} \otimes L_{\T(W)}(t_m) \otimes \Z_l \right) \oplus o(W'_2(X)) \\
 < & \, \, \left(\bigoplus_{j=0}^{k_{l,1}} \Y_{l,j,1} \otimes \Omega^{j}\right)^* \otimes \dots \otimes \left(\bigoplus_{j=0}^{k_{l,{n_l}}} \Y_{l,j,{n_l}} \otimes \Omega^{j}\right)^*  \otimes \Omega^{m_l}.
\end{align*}
Therefore,
\begin{align*}
&o(W'(\Omega))\\
 < & \bigoplus_{i=0, i \neq l}^N \left(\left(\bigoplus_{j=0}^{k_{i,1}} \Y_{i,j,1} \otimes \Omega^{j}\right)^* \otimes \dots \otimes \left(\bigoplus_{j=0}^{k_{i,{n_i}}} \Y_{i,j,{n_i}} \otimes \Omega^{j}\right)^*  \otimes \Omega^{m_i} \otimes \Z_i \right)\\
\oplus &\left(\left(\bigoplus_{j=0}^{k_{l,1}} \Y_{l,j,1} \otimes \Omega^{j}\right)^* \otimes \dots \otimes \left(\bigoplus_{j=0}^{k_{l,{n_l}}} \Y_{l,j,{n_l}} \otimes \Omega^{j}\right)^*  \otimes \Omega^{m_l} \otimes l_{\Z_l}(z) \right)\\
\oplus & \left(\left(\bigoplus_{j=0}^{k_{l,1}} \Y_{l,j,1} \otimes \Omega^{j}\right)^* \otimes \dots \otimes \left(\bigoplus_{j=0}^{k_{l,{n_l}}} \Y_{l,j,{n_l}} \otimes \Omega^{j}\right)^*  \otimes \Omega^{m_l}  \right)\\
\leq {}&{} o(W(\Omega)).
\end{align*}

\medskip

\underline{\textbf{2) $k(o(W'(\Omega))) < \vartheta(o(W(\Omega)))$.}}\\
We know that $o(W(\Omega)) \geq \Omega$, hence $\vartheta(o(W(\Omega)))$ is an epsilon number. So from Lemmas \ref{coefficients under multiplicative sum and product} and \ref{coefficients under *} the claim follows if $\Y_{i,j,k}$, $\Z_i$, $l_{\T(W)}(t_i)$ and $l_{\T(W)}(t^{i,j}_k)$ are all smaller than $\vartheta(o(W(\Omega)))$. We know that this is true for $l_{\T(W)}(t_i)$ and $l_{\T(W)}(t^{i,j}_k)$, by the sub-induction hypothesis. Furthermore, we have $\Y_{i,j,k}, \Z_k \leq k(o(W(\Omega))) < \vartheta(o(W(\Omega)))$ using Lemma \ref{coefficients under multiplicative sum and product}.
 
\end{proof}

\begin{theorem}\label{Computation upper bound for T(X**)} If $W(X)={X^*}^*$, then $\T(W)$ is a $\wpo$ and $o(\T(W))\leq \vartheta (o(W(\Omega))) $ $= \vartheta\left(\Omega^{\Omega^{\Omega^{\omega}}}\right)$.
\end{theorem}
\begin{proof}
We show that $L(t)$ is a $\wpo$ and $l(t) < \vartheta\left(\Omega^{\Omega^{\Omega^{\omega}}}\right)$ for every $t$ in $\T(W)$  by induction on the complexity of $t$. The theorem then follows from Theorem \ref{dejonghandparikh2}. If $t=\circ$, then left-set $L(t)$ is the empty $\wpo$ and $l(t)=0<\vartheta\left(\Omega^{\Omega^{\Omega^{\omega}}}\right)$.  Assume now that $t=\circ((t_1^1,\dots,t^1_{n_1}),\dots,(t_1^k,\dots,t_{n_k}^k))$. From the induction hypothesis, we know that $L(t^i_j)$ are $\wpo$'s and $l(t^i_j)< \vartheta (o(W(\Omega)))$. Assume \[s= \circ((s^1_1,\dots,s^1_{m_1}),\dots,(s_1^l,\dots,s_{m_l}^l)).\] Then $t \leq s$ iff $t\leq s_j^i$ for certain $i$ and $j$ or
\begin{align*}
((t_1^1,\dots,t^1_{n_1}),\dots,(t_1^k,\dots,t_{n_k}^k)){\leq^*}^*((s^1_1,\dots,s^1_{m_1}),\dots,(s_1^l,\dots,s_{m_l}^l)).
\end{align*}
Hence, $s \in L(t)$ iff $s_j^i \in L(t)$ for every $i$ and $j$ and one of the following holds
\begin{enumerate}
\item[1.] $(t_1^1,\dots,t_{n_1}^1) \not\leq^* (s^i_1,\dots,s^i_{n_i})$  for every $i$,
\item[2.] there exists an index $l_1$ such that $(t_1^1,\dots,t_{n_1}^1) \not\leq^* (s^i_1,\dots,s^i_{m_i})$  for every $i<l_1$, $(t_1^1,\dots,t_{n_1}^1) \leq^* (s^{l_1}_1,\dots,s^{l_1}_{m_{l_1}})$ and $(t_1^2,\dots,t_{n_2}^2) \not\leq^* (s^i_1,\dots,s^i_{m_i})$ for every $i> l_1$,
\item[3.] there exist indices $l_1<l_2$ such that $(t_1^1,\dots,t_{n_1}^1) \not\leq^* (s^i_1,\dots,s^i_{m_i})$  for every $i<l_1$, $(t_1^1,\dots,t_{n_1}^1) \leq^* (s^{l_1}_1,\dots,s^{l_1}_{m_{l_1}})$, $(t_1^2,\dots,t_{n_2}^2) \not\leq^* (s^i_1,\dots,s^i_{m_i})$ for every $l_1< i < l_2$, $(t_1^2,\dots,t_{n_2}^2) \leq^* (s^{l_2}_1,\dots,s^{l_2}_{m_{l_2}})$ and $(t_1^3,\dots,t_{n_3}^3) \not\leq^* (s^i_1,\dots,s^i_{m_i})$ for every $i>l_2 $,
\item[] $\dots$
\item[k.] there exist indices $l_1<\dots<l_{k-1}$ such that $(t_1^1,\dots,t_{n_1}^1) \not\leq^* (s^i_1,\dots,s^i_{m_i})$  for every $i<l_1$, $(t_1^1,\dots,t_{n_1}^1) \leq^* (s^{l_1}_1,\dots,s^{l_1}_{m_{l_1}})$, $(t_1^2,\dots,t_{n_2}^2) \not\leq^* (s^i_1,\dots,s^i_{m_i})$ for every $l_1< i < l_2$, $(t_1^2,\dots,t_{n_2}^2) \leq^* (s^{l_2}_1,\dots,s^{l_2}_{m_{l_2}})$, ...,  $(t_1^{k-1},\dots,t_{n_{k-1}}^{k-1}) \leq^* (s^{l_{k-1}}_1,\dots,s^{l_{k-1}}_{m_{l_{k-1}}})$ and $(t_1^k,\dots,t_{n_k}^k) \not\leq^* (s^i_1,\dots,s^i_{m_i})$ for every $i>l_{k-1}$.
\end{enumerate}
\medskip
Now, $(t_1^i,\dots,t_{n_i}^i) \not\leq^* (s^j_1,\dots,s^j_{m_j})$ is valid iff one of the following holds
\begin{enumerate}
\item[1.] $s^j_r \in L(t^i_1)$ for every $r$,
\item[$2$.] there exists an index $r^{i,j}_1$ such that $s^j_r \in L(t^i_1)$ for every $r<r^{i,j}_1$, $t^i_1\leq s^j_{r^{i,j}_1}$ and $s^j_r\in L(t^i_2)$ for every $r>r^{i,j}_1$,
\item[] $\dots$
\item[\text{$n_i$}.] there exist indices $r^{i,j}_1<\dots <r^{i,j}_{n_{i}-1}$ such that $s^j_r \in L(t^i_1)$ for every $r<r^{i,j}_1$, $t^i_1\leq s^j_{r^{i,j}_1}$, $s^j_r\in L(t^i_2)$ for every $r^{i,j}_2> r>r^{i,j}_1$,..., $t^i_{n_{i}-1}\leq s^j_{r^{i,j}_{n_i-1}}$ and $s^j_r \in L(t^i_{n_i})$ for every $r>r^{i,j}_{n_i-1}$.
\end{enumerate}
Define
\begin{align*}
W'(X):={}&{}\bigoplus_{i=1}^k  \left (\left ( \bigoplus_{j=1}^{n_1} L(t^1_1)^* \otimes X \otimes \dots \otimes X \otimes L(t^1_j)^*\right )^*  \otimes X^* \otimes  \dots\right. \\
&\left.  \otimes \ X^* \otimes\left ( \bigoplus_{j=1}^{n_i} L(t^i_1)^* \otimes X \otimes \dots \otimes X \otimes L(t^i_j)^*\right )^* \right ).
\end{align*}

If $s=\circ((s^1_1,\dots,s^1_{m_1}),\dots,(s_1^l,\dots,s_{m_l}^l))$ is an element in $L(t)$, we can interpret $((s^1_1,\dots,$ $s^1_{m_1}),\dots,(s_1^l,\dots,s_{m_l}^l))$ as an element of $W'(L(t))$. Denote this interpretation as $f(\times s) =w(s_1,\dots,s_n)$, with $\{s_1,\dots,s_n\} \subseteq \{s^1_1,\dots,s_{m_l}^l\} \subseteq L(t)$.\\
For example assume that case 2. holds, meaning that there exists an index $l_1$ such that $(t_1^1,\dots,t_{n_1}^1) \not\leq^* (s^i_1,\dots,s^i_{m_i})$  for every $i<l_1$, $(t_1^1,\dots,t_{n_1}^1) \leq^* (s^{l_1}_1,\dots,s^{l_1}_{m_{l_1}})$ and $(t_1^2,\dots,t_{n_2}^2) \not\leq^* (s^i_1,\dots,s^i_{m_i})$ for every $i> l_1$. Then assume that for $(t_1^1,\dots,t_{n_1}^1) $ $\not\leq^* (s^1_1,\dots,s^1_{m_1})$ subcase $1.$ holds, but for $(t_1^1,\dots,t_{n_1}^1) \not\leq^* (s^i_1,\dots,s^i_{m_i})$ with $1<i<l_1$ subcase 2. holds. Additionally, assume that for every $i>l_1$ subcase 2. holds if we look to $(t_1^2,\dots,t_{n_2}^2) \not\leq^* (s^i_1,\dots,s^i_{m_i})$. In this case, $f(\times s) = w(s_1,\dots,s_n)$ is equal to
\begin{align*}
&\left(
\left( \ \ (s^1_1,\dots, s^1_{m_1}),\left(  (s^2_1,\dots,s^2_{r^{1,2}_1 -1}),  s^2_{r^{1,2}_1 }, (s^2_{r^{1,2}_1 +1}, \dots, s^2_{m_2}) \right) , \dots,\right.\right.  \\
 & \ \ \ \ \ \ \left.
 \left(  (s^{l_1-1}_1,\dots,s^{l_1-1}_{r^{1,{l_1-1}}_1 -1}),  s^{l_1-1}_{r^{1,{l_1-1}}_1 }, (s^{l_1-1}_{r^{1,{l_1-1}}_1 +1}, \dots, s^{l_1-1}_{m_{l_1-1}}) \right) \right),\\
& \ \ \ (s^{l_1}_1,\dots,s^{l_1}_{m_{l_1}}),\\
& \ \ \ \left(   \left(  (s^{l_1+1}_1,\dots,s^{l_1+1}_{r^{2,{l_1+1}}_1 -1}),  s^{l_1+1}_{r^{2,{l_1+1}}_1 }, (s^{l_1+1}_{r^{2,{l_1+1}}_1 +1}, \dots, s^{l_1+1}_{m_{l_1+1}})\right), \dots,\right.\\
& \ \ \ \  \ \  \left.\left. \left(  (s^{l}_1,\dots,s^{l}_{r^{2,{l}}_1 -1}),  s^{l}_{r^{2,{l}}_1 }, (s^{l}_{r^{2,{l}}_1 +1}, \dots, s^{l}_{m_{l}})\right)
\right)
\right),
\end{align*}
where $w(\cdot,\dots,\cdot)$ is equal to
\begin{align*}
&\left(
\left( \ \ (s^1_1,\dots, s^1_{m_1}),\left(  (s^2_1,\dots,s^2_{r^{1,2}_1 -1}),  \cdot , (s^2_{r^{1,2}_1 +1}, \dots, s^2_{m_2}) \right) , \dots,\right.\right.  \\
 & \ \ \ \ \ \ \left.
 \left(  (s^{l_1-1}_1,\dots,s^{l_1-1}_{r^{1,{l_1-1}}_1 -1}),  \cdot , (s^{l_1-1}_{r^{1,{l_1-1}}_1 +1}, \dots, s^{l_1-1}_{m_{l_1-1}}) \right) \right),\\
& \ \ \ (\cdot, \dots, \cdot),\\
& \ \ \ \left(   \left(  (s^{l_1+1}_1,\dots,s^{l_1+1}_{r^{2,{l_1+1}}_1 -1}), \cdot , (s^{l_1+1}_{r^{2,{l_1+1}}_1 +1}, \dots, s^{l_1+1}_{m_{l_1+1}})\right), \dots,\right.\\
& \ \ \ \  \ \  \left.\left. \left(  (s^{l}_1,\dots,s^{l}_{r^{2,{l}}_1 -1}),  \cdot,  (s^{l}_{r^{2,{l}}_1 +1}, \dots, s^{l}_{m_{l}})\right)
\right)
\right),
\end{align*}

Using the just-described characterization of the set $L(t)$, one can see that the inequality $f(\times s) \leq_{W'(\T(W))} f(\times s')$ yields $\times s\leq_{W(\T(W))} \times s'$ for every $s,s' \in L(t)\backslash \{\circ\}$. A similar, but more detailed, argument can be found in the proof of Theorem \ref{computation upper bound T(W) with W(X)=B(X)}.

\medskip

Define now a mapping $g$ from $L(t)$ into $\T(W')$ in the following recursive way. Let $g(\circ)$ be $\circ$. Assume now that $s=\circ((s^1_1,\dots,s^1_{m_1}),\dots,(s_1^l,\dots,s_{m_l}^l))\in L(t)$ and that $g(s^i_j)$ is already defined for every $i$ and $j$. If $f(\times s) = w(s_1,\dots,s_n)$, let $g(s)$ be the element $\circ[w(g(s_1),\dots,g(s_n))]$ in $\T(W')$. If $s$ is the example as before, then $g(s)$ is defined as
\begin{align*}
\circ&\left(
\left( \ \ (s^1_1,\dots, s^1_{m_1}),\left(  (s^2_1,\dots,s^2_{r^{1,2}_1 -1}), g( s^2_{r^{1,2}_1 }), (s^2_{r^{1,2}_1 +1}, \dots, s^2_{m_2}) \right) , \dots,\right.\right.  \\
 & \ \ \ \ \ \ \left.
 \left(  (s^{l_1-1}_1,\dots,s^{l_1-1}_{r^{1,{l_1-1}}_1 -1}), g( s^{l_1-1}_{r^{1,{l_1-1}}_1 }), (s^{l_1-1}_{r^{1,{l_1-1}}_1 +1}, \dots, s^{l_1-1}_{m_{l_1-1}}) \right) \right),\\
& \ \ \ \left(g(s^{l_1}_1),\dots,g(s^{l_1}_{m_{l_1}}) \right),\\
& \ \ \ \left(   \left(  (s^{l_1+1}_1,\dots,s^{l_1+1}_{r^{2,{l_1+1}}_1 -1}), g( s^{l_1+1}_{r^{2,{l_1+1}}_1 }), (s^{l_1+1}_{r^{2,{l_1+1}}_1 +1}, \dots, s^{l_1+1}_{m_{l_1+1}})\right), \dots,\right.\\
& \ \ \ \  \ \  \left.\left. \left(  (s^{l}_1,\dots,s^{l}_{r^{2,{l}}_1 -1}),  g(s^{l}_{r^{2,{l}}_1 }), (s^{l}_{r^{2,{l}}_1 +1}, \dots, s^{l}_{m_{l}})\right)
\right)
\right).
\end{align*}

Is $g$ a quasi-embedding? We claim that $g(s) \leq_{\T(W')} g(s')$ implies $s \leq_{\T(W)} s'$ by induction on the sum of the complexities of $s$ and $s'$. If $s'= \circ$ or $s = \circ$, then this is trivial. Let 
\begin{align*}
s&=\circ((s^1_1,\dots,s^1_{m_1}),\dots,(s_1^l,\dots,s_{m_l}^l)),\\
s'&=\circ((s'^1_1,\dots,s'^1_{p_1}),\dots,(s'^{r}_1,\dots,s'^{r}_{p_{r}}))
\end{align*}
and assume that $g(s) \leq g(s')$, $f(\times s) = w(s_1,\dots,s_n)$ and $f(\times s') = w'(s'_1,\dots,s'_{n'})$. Then either $g(s) \leq g({s'}_i)$ for a certain $i$ or $\times g(s) \leq_{W'(\T(W'))}  \times g(s')$. In the former case, the induction hypothesis yields $s \leq {s'}_i \leq s'$. In the latter case, we have $w(g(s_1),\dots,g(s_n)) \leq_{W'(\T(W))} w'(g(s'_1),\dots,g(s'_{n'}))$. The induction hypothesis yields that $g$ is a quasi-embedding from the set $\{s_1,\dots,s_n, s'_1,\dots, s'_{n'}\}$ to $\T(W')$. Hence, the Lifting Lemma implies $f(\times s) = w(s_1,\dots,s_n) \leq_{W'(\T(W))} w'(s'_1,\dots,s'_{n'}) =f(\times s') $. From the properties of $f$, we obtain $\times s \leq_{W(\T(W))} \times s'$, hence $s \leq_{\T(W)} s'$.
\medskip

Now, Lemma \ref{quasi-embedding} and Theorem \ref{Computation upper bound for T(sum(sum X^* X))2} yield that $L(t)$ is a $\wpo$ and
\[o(L(t))\leq o(\T(W')) \leq \vartheta(o(W'(\Omega))).\]
So we can end this proof if we can show that $\vartheta(o(W'(\Omega)))<\vartheta\left(\Omega^{\Omega^{\Omega^{\omega}}}\right)$. For notational convenience, we write $\Y$ instead of $o(\Y)$ and $\Y^*$ instead of $o(\Y^*)$ for $\wpo$'s $\Y$.
From the induction hypothesis, we know that $l(t^i_j)<\vartheta\left(\Omega^{\Omega^{\Omega^{\omega}}}\right) < \Omega$, hence $l(t^i_1)^* \otimes \Omega \otimes \dots \otimes \Omega \otimes l(t^i_j)^*<\Omega^n$ for a certain finite $n$. Therefore, 
\begin{align*}
\left (\bigoplus_{j=1}^{n_i} l(t^i_1)^* \otimes \Omega \otimes \dots \otimes \Omega \otimes l(t^i_j)^*\right) +1 < \Omega^\omega
\end{align*}
and 
\begin{align*}
\left (\bigoplus_{j=1}^{n_i} l(t^i_1)^* \otimes \Omega \otimes \dots \otimes \Omega \otimes l(t^i_j)^*\right)^* < \Omega^{\Omega^{\Omega^{\omega}}}.
\end{align*}
From this it follows that $o(W'(\Omega)) < \Omega^{\Omega^{\Omega^{\omega}}}$. Now
\begin{align*}
o(W'(\Omega)) &= \bigoplus_{i=1}^k  \left (\left ( \bigoplus_{j=1}^{n_1} l(t^1_1)^* \otimes \Omega \otimes \dots \otimes \Omega \otimes l(t^1_j)^*\right )^*  \otimes \Omega^* \otimes \dots\right. \\
&\left.  \otimes \ \Omega^* \otimes \left ( \bigoplus_{j=1}^{n_i} l(t^i_1)^* \otimes \Omega \otimes \dots \otimes \Omega \otimes l(t^i_j)^*\right )^* 
\right )\\
&= \bigoplus_{i=1}^k  \left (\Omega^{\omega (i-1)}\otimes
\left ( \bigoplus_{j=1}^{n_1} \Omega^{j-1} \otimes l(t^1_1)^* \otimes \dots \otimes l(t^1_j)^*\right )^*  \otimes  \dots\right. \\
&\left.  \otimes  \left ( \bigoplus_{j=1}^{n_i} \Omega^{j-1}\otimes l(t^i_1)^*  \otimes \dots \otimes l(t^i_j)^*\right )^* 
\right ).\\
\end{align*}
Using Lemmas \ref{coefficients under multiplicative sum and product} and \ref{coefficients under *}, $k(o(W'(\Omega))) < \vartheta\left(\Omega^{\Omega^{\Omega^{\omega}}}\right)$ is valid if 
\[l(t^i_1)^*  \otimes \dots \otimes l(t^i_j)^* < \vartheta\left(\Omega^{\Omega^{\Omega^{\omega}}}\right).\]
This is true because $l(t^i_j) < \vartheta\left(\Omega^{\Omega^{\Omega^{\omega}}}\right)$ and $\vartheta\left(\Omega^{\Omega^{\Omega^{\omega}}}\right)$ is an epsilon number.
   \end{proof}

These results can be generalized to the following theorem. We do not give the proof as it is very technical and it will not provide new insights because it follows the same procedures as in Theorems \ref{Computation upper bound for T(sum(sum X^* X))2} and \ref{Computation upper bound for T(X**)}.

\begin{theorem}\label{computation upper bound T(W) with W(X)=X***} Suppose that $W\in Map$ consists only of $\cdot$, $+$, $\times$, ${}^*$ and countable $\wpo$'s $\X_i$. Then $\T(W)$ is a $\wpo$ and $o(\T(W)) \leq \vartheta(o(W(\Omega)))$.
\end{theorem}

One can also show that $\vartheta(o(W(\Omega)))$ is a lower bound for $o(\T(W))$ if $o(W(\Omega))\geq \Omega^3$. However, the purpose of this section is to proof theorems that are needed in section \ref{sec:an order-theoretic approach of the Howard-Bachmann ordinal number}. Therefore, we skip the proof of the lower bound.


\section{An order-theoretic approach of the Howard-Bach\-mann ordinal}\label{sec:an order-theoretic approach of the Howard-Bachmann ordinal number}
The previous section yields 
\[o(\T(\cdot^{\overbrace{*\dots*}^{n}})) = \vartheta(\Omega_{2n-1}[\omega]).\]
Therefore, the tree-structures $\T(\cdot^{\overbrace{*\dots*}^{n}})$ give rise to representations of countable ordinals strictly below the Howard-Bachman ordinal and the `limit' of these structures will be equal to this famous ordinal.
But what do we mean by the `limit' of these structures? In some sense, the set of binary trees is the limit of an iteration of the $*$-operator. Hence, one can expect that $o(\T(\mathcal{B}(\cdot))) = \sup_{n<\omega} \vartheta(\Omega_{2n-1}[\omega]) = \vartheta(\varepsilon_{\Omega+1})$. In this section, we will prove that this is indeed the case. This result yields that the Howard-Bachmann ordinal can be represented as a tree-structure using binary trees, or more specifically, as the $\wpo$ $(\overline{\mathbb{T}}_2,\leq_{gap})$.

\begin{theorem}\label{computation upper bound T(W) with W(X)=B(X)} $\T(\mathcal{B}(\cdot))$ is a $\wpo$ and $o(\T(\mathcal{B}(\cdot)))\leq \vartheta(\varepsilon_{\Omega+1})$.
\end{theorem}
\begin{proof}
We prove that $L(t)$ is a $\wpo$ and $l(t) < \vartheta(\varepsilon_{\Omega+1})$ for every $t$ in $\T(\mathcal{B}(\cdot))$ by induction on the complexity of $t$. The theorem then follows from Theorem \ref{dejonghandparikh2}. If $t=\circ$, then $L(t)$ is the empty $\wpo$ and $l(t)=0<\vartheta(\varepsilon_{\Omega+1})$. 

\medskip

Let $B(t_1,\dots,t_n)$ be an element of $\mathcal{B}(\T(\mathcal{B}(\cdot)))$. If we write $B(t_1,\dots,t_n)$, we mean that the leaf-labels are elements of $\{t_1,\dots,t_n\}$. If it is clear from the context, we sometimes write $B$ instead of $B(t_1,\dots,t_n)$. If $B(t_1,\dots,t_n)$ is a tree of height zero with leaf-label $t_i$, define $W_B(X)$ as the partial ordering $\mathcal{B}(L_{\T(\mathcal{B}(\cdot))}(t_i))$. Remark that $\cdot$ does not occur in $W_B$. If $B(t_1,\dots,t_n)$ is a tree with immediate subtrees $B_1$ and $B_2$, define $W_{B}(X) = W_{B(t_1,\dots,t_n)}(X)$ as \[(W_{B_1}(X) + W_{B_2}(X))^* \times X.\]
We prove by induction on the height of the tree $B$ that there exists a mapping $g_B$ from 
\[L_B :=  \{D(d_1,\dots,d_k) \in \mathcal{B}(\T(\mathcal{B}(\cdot))): B(t_1,\dots,t_n) \not\leq_{\mathcal{B}(\T(\mathcal{B}(\cdot)))} D(d_1,\dots,d_k)\}\]
to the partial ordering $W_B(\T(\mathcal{B}(\cdot)))$ such that $g_B$ is a quasi-embedding and if $g_B(D) = w(d'_1,\dots,d'_m)$, with $w$ an element in $W_{B}$ and $d'_1,\dots,d'_m \in \T(\mathcal{B}(\cdot))$, then $\{d'_1,\dots,d'_m\} \subseteq \{d_1,\dots,d_k\}$.

\medskip 

\textit{i) $height(B) =0$.}\\
Let $B(t_1,\dots,t_n)$ be a tree with one node and leaf-label $t_i$. Then $D(d_1,\dots,d_k)\in L_{B}$ iff  $d_j \in L(t_i)$ for every $j$. Define then the element $g_{B}(D(d_1,\dots,d_k))$ as $D(d_1,\dots,d_k)\in \mathcal{B}(L_{\T(\mathcal{B}(\cdot))}(t_i))=W_{B}(\T(\mathcal{B}(\cdot)))$. If we write $g_B(D)$ as $w(d'_1,\dots,d'_m)$, then $m=0$, hence one can show easily that the desired properties of $g_{B}$ are valid.

\medskip

\textit{ii) $height(B) >0$.}\\
Let $B(t_1,\dots,t_n)$ be a binary tree with immediate subtrees $B_1$ and $B_2$. By the induction hypothesis, there exist functions $g_{B_1}$ and $g_{B_2}$ with the wanted properties.  
Now, Pick an arbitrary $D(d_1,\dots,d_k) \in \mathcal{B}(\T(\mathcal{B}(\cdot)))$. Then $B(t_1,\dots,t_n) \not\leq_{\mathcal{B}(\T(\mathcal{B}(\cdot)))} D(d_1,\dots,d_k)$ is valid iff one of the following holds
\begin{enumerate}
\item $D(d_1,\dots,d_k)$ is a binary tree of height $0$ with label $d_i$,
\item $D(d_1,\dots,d_k)$ is a tree of height strictly larger than $0$ with immediate subtrees $D_1$ and $D_2$, $B(t_1,\dots,t_n) \not\leq_{\mathcal{B}(\T(\mathcal{B}(\cdot)))} D_i$ for $i=1,2$ and one of the following occurs
\begin{enumerate}
\item $B_1 \not\leq_{\mathcal{B}(\T(\mathcal{B}(\cdot)))} D_1$,
\item $B_1 \leq_{\mathcal{B}(\T(\mathcal{B}(\cdot)))} D_1$ and $B_2 \not\leq_{\mathcal{B}(\T(\mathcal{B}(\cdot)))} D_2$.
\end{enumerate}
\end{enumerate}

Because $B(t_1,\dots,t_n) \not\leq_{\mathcal{B}(\T(\mathcal{B}(\cdot)))} D_i$ for $i=1,2$, we can also use the above case-study for the trees $D_1$ and $D_2$. This leads us to the following definition. 
Choose an arbitrary $D(d_1,\dots,d_k) \in L_B$. Define $E_0$ and $F_0$ as $D(d_1,\dots,d_k)$. Assume that we have $E_i$ and $F_i$ for a certain $i$ as elements of $\mathcal{B}(\T(\mathcal{B}(\cdot)))$. If $F_i$ is a tree of height strictly larger than $0$, we want to define $E_{i+1}$ and $F_{i+1}$. Suppose that $F_i^1$ and $F_i^2$ are the immediate subtrees of $F_i$. Then define $E_{i+1}$ as
\begin{align}
\label{definition E}
\left\{\begin{array}{ll}
F_i^1 & \text{in case that $2.(a)$ holds if we look to the condition $B \not\leq_{\mathcal{B}(\T(\mathcal{B}(\cdot)))}F_i$}\\
F_i^2&\text{in case that $2.(b)$ holds if we look to the condition $B\not\leq_{\mathcal{B}(\T(\mathcal{B}(\cdot)))}F_i$.}
\end{array}
\right.
\end{align}
Let $x_{i+1}$ be the number $j$ such that $E_{i+1}  = F_{i}^j$ and let $F_{i+1}$ be $F_i^{3-x_{i+1}}$, the \textit{other} immediate subtree of $F_i$. 
From this definition, we obtain a finite sequence $E_0,E_1,\dots,E_p,F_p$ with $E_0=D$ and $F_p$ a tree of height $0$. Therefore, $F_p$ consists of only one node with a label, let us say, $s$. Note that $s$ is also a leaf-label of the tree $D$.
Define now $g_{B}(D)$ as follows using the fact that we have $g_{B_1}$ and $g_{B_2}$:
\[g_{B}(D) := ((g_{B_{x_1}}(E_1),\dots,g_{B_{x_p}}(E_p)), s ).\]
Note that $B_{x_i} \not\leq E_i$, which means that $E_i \in L_{B_{x_i}}$. So $g_{B_{x_i}}(E_i)$ is well-defined, hence $g_B(D) \in W_{B}(\T(\mathcal{B}(\cdot)))$. Does $g$ satisfies the desired properties?

\medskip

First, we already noted that $s$ is a leaf-label of $D(d_1,\dots,d_k)$. Secondly, if $g_{B_{x_i}}(E_i)$ is equal to $w_i(s^i_1,\dots,s^i_{n_i})$, then by the induction hypothesis and the fact that $E_i$ is a subtree of $D$, we obtain $\{s^i_1,\dots,s^i_{n_i}\} \subseteq \{d_1,\dots,d_k\}$. Hence, if $g_B(D) = w(d'_1,\dots,d'_m) \in W_{B}(\T(\mathcal{B}(\cdot)))$, then $\{d'_1,\dots,d'_m\} \subseteq \{d_1,\dots,d_k\}$. Now we want to prove that $g$ is a quasi-embedding.
Let $\overline{E_0},\overline{E_1},\dots,\overline{E_{q}},\overline{F_{q}}$ and $y_1,\dots,y_q$ be the finite sequences forthcoming from definition (\ref{definition E}), the definitions of $x_{i+1}$ and $F_{i+1}$, but now starting with $ \overline{D}(\overline{d_1},\dots, \overline{d_l}) \in L_B$. Denote the label of the tree $\overline{F_{q}}$ of height zero as $\overline{s}$. Assume furthermore that 
\begin{align}
\nonumber
&{}\ g_B(D) = ((g_{B_{x_1}}(E_1),\dots,g_{B_{x_p}}(E_p)), s ) \\
\label{seconddesiredpropertyg_B} \leq_{W_B(\T(\mathcal{B}(\cdot)))}   &{}\ 
g_B(\overline{D})  = ((g_{B_{y_{1}}}(\overline{E_{1}}),\dots,g_{B_{y_q}}(\overline{E_q})), \overline{s} ).
\end{align}
We show that inequality $D(d_1,\dots,d_k) \leq_{\mathcal{B}(\T(\mathcal{B}(\cdot)))} \overline{D}(\overline{d_1},\dots, \overline{d_l})$ holds by induction on $q$.

\medskip

From (\ref{seconddesiredpropertyg_B}), we obtain $s \leq_{\T(\mathcal{B}(\cdot))} \overline{s}$. 
Furthermore, there exist indices $1\leq i_1< \dots < i_p \leq q$ such that 
$g_{B_{x_j}}(E_j) \leq  g_{B_{y_{i_j}}}(\overline{E_{i_j}})$.
Because the left hand side of this inequality is in $W_{B_{x_j}}(\T(\mathcal{B}(\cdot)))$ and the right hand side in $W_{B_{y_{i_j}}}(\T(\mathcal{B}(\cdot)))$, we obtain $x_j  = y_{i_j}$ for every $j$. Furthermore, $E_j \leq_{\mathcal{B}(\T(\mathcal{B}(\cdot)))} \overline{E_{i_j}}$ for every $j$, because $g_{B_{x_j}}$ is a quasi-embedding.\\
If $q=0$, then also $p=0$. Therefore, $D$ is a tree of height zero with leaf-label $s$ and $\overline{D}$ is a tree of the same height with leaf-label $\overline{s}$. Hence $D \leq_{\mathcal{B}(\T(\mathcal{B}(\cdot)))}  \overline{D}$. Now let $q>0$. By construction, 
\begin{align*}
g_{B}(F_1) &= ((g_{B_{x_2}}(E_2),\dots,g_{B_{x_p}}(E_p)), s ),\\
g_B(\overline{F_{i_1}}) &= ((g_{B_{y_{i_1+1}}}(\overline{E_{i_1+1}}),\dots,g_{B_{y_q}}(\overline{E_q})), \overline{s} ).
\end{align*} 
Because $g_B(D) \leq_{W_B(\T(\mathcal{B}(\cdot)))} g_B(\overline{D})$ and $g_{B_{x_1}}(E_1)$ is mapped onto $g_{B_{y_{i_1}}}(\overline{E_{i_1}})$, we obtain $g_B(F_1) \leq_{W_B(\T(\mathcal{B}(\cdot)))} g_B(\overline{F_{i_1}})$. Hence, by the sub-induction hypothesis on $q$, we have the inequality $F_1 \leq_{\mathcal{B}(\T(\mathcal{B}(\cdot)))} \overline{F_{i_1}}$. We also know that $E_1 \leq_{\mathcal{B}(\T(\mathcal{B}(\cdot)))} \overline{E_{i_1}}$ and $x_1=y_{i_1}$. If $x_1=1$, then $E_1$ is the left-immediate subtree of $F_0=D$ and $\overline{E_{i_1}}$ is the left-immediate subtree of $\overline{F_{i_1-1}}$. Furthermore, $F_1$ is the right-immediate subtree of $F_0=D$ and $\overline{F_{i_1}}$ is the right-immediate subtree of $\overline{F_{i_1-1}}$. We conclude that $D = F_0 \leq_{\mathcal{B}(\T(\mathcal{B}(\cdot)))} \overline{F_{i_1-1}} \leq_{\mathcal{B}(\T(\mathcal{B}(\cdot)))} \overline{F_0} = \overline{D}$. The same argument holds for $x_1=2$. Therefore, $g_B$ is a quasi-embedding.

\medskip

Assume $t=\circ[B(t_1,\dots,t_n)] \in \T(\mathcal{B}(\cdot))$ with $B(t_1,\dots,t_n)$ a binary tree in the partial ordering $\mathcal{B}(\T(\mathcal{B}(\cdot)))$ and assume that $L(t_i)$ are $\wpo$'s and $l(t_i)<\vartheta(\varepsilon_{\Omega+1})$. We want to prove that $L(t)$ is a $\wpo$ and $l(t)<\vartheta(\varepsilon_{\Omega+1})$.\\
First of all, we define a quasi-embedding $f$ from $L(t)$ into $\T(W_{B})$. First note that $d=\circ[D(d_1,\dots,d_k)] \in L(t)$ iff $d_i \in L(t)$ and $D(d_1,\dots,d_k) \in L_B$. Define $f(\circ)$ as $\circ$. Suppose $d=\circ[D(d_1,\dots,d_k)] \in L(t)$ and that $f(d_1),\dots,f(d_k)$ are already defined. If $g_{B}(D)=w(d'_1,\dots,d'_m)\in W_{B}(\T(\mathcal{B}(\cdot)))$ with $\{d'_1,\dots,d'_m\} \subseteq \{d_1,\dots,d_k\}$, define $f(d)$ as $\circ[w(f(d'_1),\dots,f(d'_m))]\in \T(W_{B})$. Now we want to prove that $f$ is a quasi-embedding.\\
Assume $f(d) \leq_{\T(W_B)} f(\overline{d})$. We prove that this implies $d\leq_{\T(\mathcal{B}(\cdot))} \overline{d}$ by induction on the sum of the complexities of $d$ and $\overline{d}$. If $d$ or $\overline{d}$ is equal to $\circ$, then this is trivial. Assume $g_B(D(d_1,\dots,d_k)) = w(d'_1,\dots,d'_m)$, $g_B(\overline{D}(\overline{d_1},\dots,\overline{d_l})) = \overline{w}(\overline{d'_1},\dots,\overline{d'_p})$ with $\{d'_1, \dots, d'_m \} \subseteq \{d_1,\dots,d_k\}$ and $\{ \overline{d'_1},\dots,\overline{d'_p} \} \subseteq \{ \overline{d_1},\dots,\overline{d_l} \}$ and 
\begin{align*}
&f(d) = f(\circ[D(d_1,\dots,d_k)]) = \circ[w(f(d'_1),\dots,f(d'_m))] \\
\leq_{\T(W_{B})} &f(\overline{d})= f(\circ[\overline{D}(\overline{d_1},\dots,\overline{d_l})]) = \circ[\overline{w}(f(\overline{d'_1}),\dots,f(\overline{d'_{p}}))].
\end{align*}
This implies either $f(d) \leq_{\T(W_{B})} f(\overline{d'_i})$ for some $i$ or 
\[w(f(d'_1),\dots,f(d'_m)) \leq_{W_B(\T(W_{B}))} \overline{w}(f(\overline{d'_1}),\dots,f(\overline{d'_{p}})).\]
In the former case, the induction hypothesis yields $d\leq_{\T(\mathcal{B}(\cdot))} \overline{d'_i} \leq_{\T(\mathcal{B}(\cdot))} \overline{d}$. 
In the latter case, we observe that the induction hypothesis implies that $f$ is a quasi-embedding from $\{d_1,\dots,d_k, \overline{d_1}, \dots, \overline{d_{l}}\}$ to $\T(W_B)$. Therefore, the Lifting Lemma brings the inequality $g_{B}(D) \leq_{W_B(\T(W_B))} g_B(\overline{D})$. Hence we have $D(d_1,\dots,d_k) \leq_{\mathcal{B}(\T(\mathcal{B}(\cdot)))} \overline{D}(\overline{d_1},\dots,\overline{d_{l}})$ and $d \leq_{\T(\mathcal{B}(\cdot))} \overline{d}$. We conclude that $f$ is a quasi-embedding.

\medskip

By Lemma \ref{quasi-embedding} and Theorem \ref{computation upper bound T(W) with W(X)=X***} we obtain that $L(t)$ is a $\wpo$ and 
\begin{align*}
o(L(t)) \leq o(\T(W_{B})) \leq \vartheta(W_{B}(\Omega)).
\end{align*}
So if $\vartheta(W_{B}(\Omega)) < \vartheta(\varepsilon_{\Omega+1})$, we can end this proof. We prove simultaneously by induction on the height of the tree $B$ that 
\begin{align*}
W_{B}(\Omega)&< \Omega_{n}[1] ,\\
k(W_{B}(\Omega))&< \vartheta( \Omega_n[1]),
\end{align*}
for a certain natural number $n$. Remark that we write $W_B(\Omega)$ instead of $o(W_B(\Omega))$ for notational convenience. From this it follows that $\vartheta(W_{B}(\Omega)) < \vartheta(\varepsilon_{\Omega+1})$.

\medskip

If the height of the tree $B$ is zero, we defined $W_{B}(X)$ as $\mathcal{B}(L_{\T(\mathcal{B}(\cdot))}(t_i))$. Because we assumed that $l(t_i)  < \vartheta(\varepsilon_{\Omega+1})$, we know that $l(t_i)< \vartheta(\Omega_m[1])< \Omega$ for a certain natural number $m \geq 2$.
Hence
\begin{align*}
k(W_{B}(\Omega))& =  k(o(\mathcal{B}(L(t_i))) ) = o(\mathcal{B}(L(t_i)))   \leq  \varepsilon_{l(t_i)  +1} \\
& < \varepsilon_{\vartheta(\Omega_m[1])+1} = \vartheta(\Omega+ \vartheta(\Omega_m[1])) \leq  \vartheta(\Omega_{m+1}[1])
\end{align*}
and
\[W_{B}(\Omega)   \leq  \varepsilon_{l(t_i)  +1} < \Omega < \Omega_{m+1}[1].\]

Assume that the height of $B$ is strictly larger than zero such that $B_1$ and $B_2$ are immediate subtrees of $B$. Because of the induction hypothesis, we know that there exists a natural number $m$ such that 
\begin{align*}
W_{B_1}(\Omega)&< \Omega_m[1] ,\\
k(W_{B_1}(\Omega))&< \vartheta( \Omega_m[1] ),\\
W_{B_2}(\Omega)&< \Omega_m[1] ,\\
k(W_{B_2}(\Omega))&< \vartheta( \Omega_m[1] ).
\end{align*}
We defined $W_B(X)$ as $(W_{B_1}(X) + W_{B_2}(X))^* \times X$, hence
\begin{align*}
W_{B}(\Omega)&< \omega^{\omega^{\Omega_m[1] \oplus \Omega_m[1] \oplus 1}} \otimes \Omega < \Omega_{n_1}[1]
\end{align*}
for a certain $n_1$ large enough. Furthermore, by Lemma \ref{coefficients under multiplicative sum and product}
\[k(W_B(\Omega)) = k\left(\omega^{\omega^{W_{B_1}(\Omega) \oplus W_{B_2}(\Omega) (\pm 1)}}\otimes \Omega\right)<\vartheta( \Omega_{n_2}[1])\]
for some $n_2$ large enough. Taking $n= \max\{n_1,n_2\}$, we can end this proof.
   \end{proof}

Before we prove that the Howard-Bachmann ordinal is a lower bound of the maximal order type of $\T(\mathcal{B}(\cdot))$, we give a lemma.

\begin{lemma}\label{alpha strikt kleiner dan theta(epsilon_Omega+1)} Suppose $\alpha< \vartheta(\varepsilon_{\Omega+1})$ and $\alpha \in P$, the set of additively closed ordinals. Then $\alpha$ can be uniquely  written as
\[\alpha =\vartheta(\beta),\]
with $\beta= 0$ or 
\[\beta=_{NF}\Omega^{\beta_1}\gamma_1+\dots+ \Omega^{\beta_n}\gamma_n< \varepsilon_{\Omega+1},\]
such that $\beta > \beta_1 > \dots > \beta_n$ and $0<\gamma_i < \Omega$ and $k(\beta)< \alpha$.
\end{lemma}
\begin{proof}
This follows from standard properties of the $\vartheta$-function. A proof of this fact can be found in an unpublished article of Buchholz.
\end{proof}

If you look closer at the proof of the next theorem, one can see how every ordinal number below $\vartheta(\varepsilon_{\Omega+1})$ can be represented as an element of $\T(\mathcal{B}(\cdot))$. 
Note that this proof can be carried out in $ACA_0$ if we have a predefined primitive recursive ordinal notation system for $\vartheta(\varepsilon_{\Omega+1})$.

\begin{theorem} \label{computation lower bound W(X)=B(X)} $o(\T(\mathcal{B}(\cdot)))\geq \vartheta(\varepsilon_{\Omega+1})$.
\end{theorem}
\begin{proof}
Define
\begin{align*}
g: \vartheta(\varepsilon_{\Omega+1}) \to \T(\mathcal{B}(\cdot))
\end{align*}
in the following recursive way. Let $g(0)$ be $\circ$. Pick an arbitrary $\alpha < \vartheta(\varepsilon_{\Omega+1})$ and assume that $g(\beta)$ is already defined for every $\beta < \alpha$. If $\alpha=_{CNF}\omega^{\alpha_1}+\dots+\omega^{\alpha_n}$ with $n\geq 2$, define $g(\alpha)$ as $\circ[B]$ with $B$ the following binary tree: 

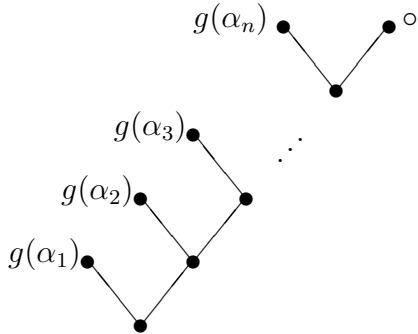
\begin{picture}(1,130)
\put(60,-4){\circle*{5}}
\put(40,20){\circle*{5}}
\put(80,20){\circle*{5}}
\put(60,44){\circle*{5}}
\put(100,44){\circle*{5}}
\put(80,68){\circle*{5}}
\put(134,85){\circle*{5}}
\put(114,109){\circle*{5}}
\put(154,109){\circle*{5}}

\put(120,55){\begin{rotate}{80} $\ddots$ \end{rotate}}

\put(60,-4){\line(4,5){20}}
\put(60,-4){\line(-4,5){20}}
\put(80,20){\line(4,5){20}}
\put(80,20){\line(-4,5){20}}
\put(100,44){\line(-4,5){20}}
\put(134,85){\line(4,5){20}}
\put(134,85){\line(-4,5){20}}

\put(10,20){$g(\alpha_1)$}
\put(30,44){$g(\alpha_2)$}
\put(50,68){$g(\alpha_3)$}
\put(80,109){$g(\alpha_{n})$}
\put(159,109){$\circ$}

  \end{picture}
\ \\ \ \\
If $\alpha<  \vartheta(\varepsilon_{\Omega+1})$ and $\alpha \in P$, we can write $\alpha$ as $\vartheta(\beta)$ as in Lemma \ref{alpha strikt kleiner dan theta(epsilon_Omega+1)}. Because every element in $K(\beta)$ is strictly smaller than $\alpha$, we can assume that $g(\gamma)$ is defined for every $\gamma \in K(\beta)$. Define $g(\alpha)$ as $\circ[f(\beta)]$, where we define the binary tree $f(\beta)$ in the following recursive way.

\medskip

Let $f(0)$ be the binary tree with one node and leaf-label $\circ$:

\begin{picture}(1,40)

\put(12,20){\circle*{5}}
\put(18,8){$\circ$}

\end{picture}

Now, let $\beta = \Omega^{\beta_1}\gamma_1+\dots + \Omega^{\beta_n} \gamma_n> \beta_1 > \dots > \beta_n\geq 0$ and assume that $f(\beta_i)$ is already defined for every $i$. Then define $f(\beta)$ as

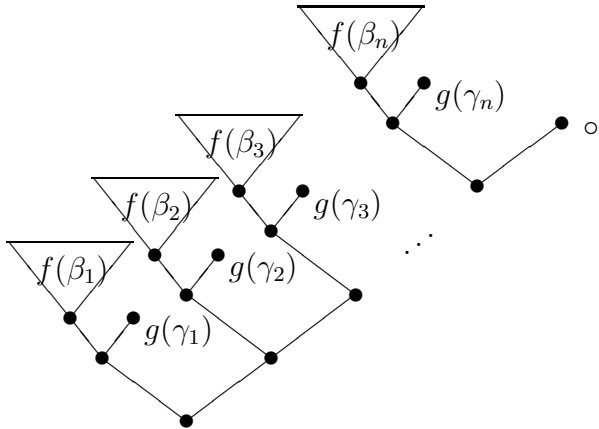
\begin{picture}(1,175)
\put(60,-4){\circle*{5}}
\put(28,20){\circle*{5}}
\put(92,20){\circle*{5}}
\put(60,44){\circle*{5}}
\put(124,44){\circle*{5}}
\put(92,68){\circle*{5}}
\put(170,85){\circle*{5}}
\put(202,109){\circle*{5}}
\put(138,109){\circle*{5}}

\put(150,55){\begin{rotate}{70} $\ddots$ \end{rotate}}

\put(60,-4){\line(4,3){30}}
\put(60,-4){\line(-4,3){30}}
\put(92,20){\line(4,3){30}}
\put(92,20){\line(-4,3){30}}
\put(124,44){\line(-4,3){30}}
\put(170,85){\line(4,3){30}}
\put(170,85){\line(-4,3){30}}

\put(44,27){$g(\gamma_1)$}
\put(16,35){\circle*{5}}
\put(40,35){\circle*{5}}
\put(28,20){\line(4,5){13}}
\put(28,20){\line(-4,5){13}}

\put(76,51){$g(\gamma_2)$}
\put(48,59){\circle*{5}}
\put(72,59){\circle*{5}}
\put(60,44){\line(4,5){13}}
\put(60,44){\line(-4,5){13}}

\put(108,75){$g(\gamma_3)$}
\put(80,83){\circle*{5}}
\put(104,83){\circle*{5}}
\put(92,68){\line(4,5){13}}
\put(92,68){\line(-4,5){13}}

\put(154,116){$g(\gamma_{n})$}
\put(126,124){\circle*{5}}
\put(150,124){\circle*{5}}
\put(138,109){\line(4,5){13}}
\put(138,109){\line(-4,5){13}}

\put(210,104){$\circ$}

\put(16,35){\line(4,5){23}}
\put(16,35){\line(-4,5){23}}
\put(-8,64){\line(1,0){48}}
\put(3,50){$f(\beta_1)$}

\put(48,59){\line(4,5){23}}
\put(48,59){\line(-4,5){23}}
\put(24,88){\line(1,0){48}}
\put(35,74){$f(\beta_2)$}

\put(80,83){\line(4,5){23}}
\put(80,83){\line(-4,5){23}}
\put(56,112){\line(1,0){48}}
\put(67,98){$f(\beta_3)$}

\put(126,124){\line(4,5){23}}
\put(126,124){\line(-4,5){23}}
\put(102,153){\line(1,0){48}}
\put(113,139){$f(\beta_{n})$}

\end{picture}
\ \\
\ \\
Note that all labels in the tree $f(\beta)$ are elements of $g(K(\beta) \cup \{0\})$. Additionally, every element in $g(K(\beta) \cup \{0\})$ is a label in the tree $f(\beta)$.\\

Is $g$ a quasi-embedding? We will show by induction on $\alpha'$ that $g(\alpha) \leq g(\alpha')$ implies $\alpha\leq \alpha'$. If $\alpha$ or $\alpha'$ is equal to zero, then this is trivial, hence we may assume that both $\alpha$ and $\alpha'$ are different from zero. There are now four cases left:\\

\textit{
a) $g(\omega^{\alpha_1}+ \dots +\omega^{\alpha_n}) \leq g(\omega^{\alpha'_1}+ \dots + \omega^{\alpha'_m})$.}\\
Then either $g(\alpha) \leq g(\alpha'_i)$ or $(g(\alpha_1),\dots,g(\alpha_m))\leq^* (g(\alpha'_1),\dots,g(\alpha'_m))$. In both cases the induction hypothesis yields $\omega^{\alpha_1}+ \dots +\omega^{\alpha_n} \leq \omega^{\alpha'_1}+ \dots + \omega^{\alpha'_m}$.\\

\textit{
b) $g(\vartheta(\beta)) \leq g(\omega^{\alpha'_1}+ \dots + \omega^{\alpha'_m})$.}\\
If $\beta=0$, then it is trivial. Assume $\beta \neq 0$. Then $g(\vartheta(\beta)) \leq g(\omega^{\alpha'_1}+ \dots + \omega^{\alpha'_m})$ is only possible if $g(\vartheta(\beta))\leq g(\alpha'_i)$ for a certain $i$. The induction hypothesis yields $\vartheta(\beta)\leq \omega^{\alpha'_1}+ \dots + \omega^{\alpha'_m} $.\\

\textit{
c) $ g( \omega^{\alpha'_1}+ \dots + \omega^{\alpha'_m})\leq g(\vartheta(\beta))$.}\\
It is impossible that $\beta=0$, because $m\geq 2$. If $0<\beta <\Omega$, then $g(\vartheta(\beta))$ is equal to

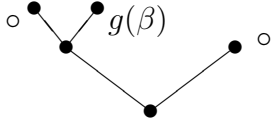
\begin{picture}(1,65)
\put(60,-4){\circle*{5}}
\put(60,-4){\line(4,3){30}}
\put(60,-4){\line(-4,3){30}}
\put(92,20){\circle*{5}}
\put(100,20){$\circ$}
\put(28,20){\circle*{5}}

\put(44,27){$g(\beta)$}
\put(16,35){\circle*{5}}
\put(40,35){\circle*{5}}
\put(28,20){\line(4,5){13}}
\put(28,20){\line(-4,5){13}}

\put(5,27){$\circ$}

\end{picture}
\ \\ \ \\
So $g( \omega^{\alpha'_1}+ \dots + \omega^{\alpha'_m})\leq g(\vartheta(\beta))$ can only occur if $g( \omega^{\alpha'_1}+ \dots + \omega^{\alpha'_m})\leq g(\beta)$ because $m\geq 2$. By the induction hypothesis, we obtain that $ \omega^{\alpha'_1}+ \dots + \omega^{\alpha'_m}\leq \beta < \vartheta(\beta)$.\\
In the case that $\beta \geq \Omega$, we have that $ g( \omega^{\alpha'_1}+ \dots + \omega^{\alpha'_m})\leq g(\gamma)$ for a certain $\gamma \in K(\beta) \cup \{0\}$ or for every $\alpha'_i$, there exists a $\gamma_i \in K(\beta) \cup \{ 0 \} $ such that $g(\alpha'_i)\leq g(\gamma_i)$. In the former case, we obtain that  $\omega^{\alpha'_1}+ \dots + \omega^{\alpha'_m}\leq \gamma \leq k(\beta)< \vartheta(\beta)$. In the latter case, we obtain that $\alpha'_i \leq \gamma_i \leq k(\beta) < \vartheta(\beta)$. Because $\vartheta(\beta)$ is an epsilon number ($\beta \geq \Omega$), we get that $ \omega^{\alpha'_1}+ \dots + \omega^{\alpha'_m} < \vartheta(\beta)$.\\

\textit{
d) $ g(\vartheta(\beta))\leq g(\vartheta(\beta'))$}\\
If $\beta'=0$, then $\beta$ must also be zero, hence $\vartheta(\beta) \leq \vartheta(\beta')$. Now assume that $\beta'>0$. $ g(\vartheta(\beta))\leq g(\vartheta(\beta'))$ is possible if $g(\vartheta(\beta)) \leq g(\gamma')$ for a certain $\gamma' \in K(\beta')\cup \{0\}$ or if $f(\beta) \leq f(\beta')$. In the former case, we obtain by the induction hypothesis that $\vartheta(\beta) \leq \gamma' \leq k(\beta') < \vartheta(\beta')$. If the latter case occurs, then for every $\gamma \in K(\beta)$, there exists a $\gamma' \in K(\beta') \cup \{0 \} $ such that $g(\gamma) \leq g(\gamma')$. Hence, $k(\beta) \leq k(\beta')< \vartheta(\beta')$. For ending the proof of $\vartheta(\beta) \leq \vartheta(\beta')$, we need to show that $\beta \leq \beta'$.\\

We proof by induction on $\delta'$ that $f(\delta) \leq f(\delta')$ implies $\delta \leq \delta'$ for every $\delta$ with $K(\delta) \subseteq K(\beta)$ and every $\delta'$ with $K(\delta') \subseteq K(\beta')$. If this is true can conclude that $\beta \leq \beta'$.\\
If $\delta'=0$ or $\delta=0$, then this is trivial. Hence we may assume that both $\delta$ and $\delta'$ are different from zero.  Assume $\delta' =_{NF} \Omega^{\delta'_1} \gamma'_1 + \dots + \Omega^{\delta'_m} \gamma'_m>0$. Then $f(\delta')$ is equal to

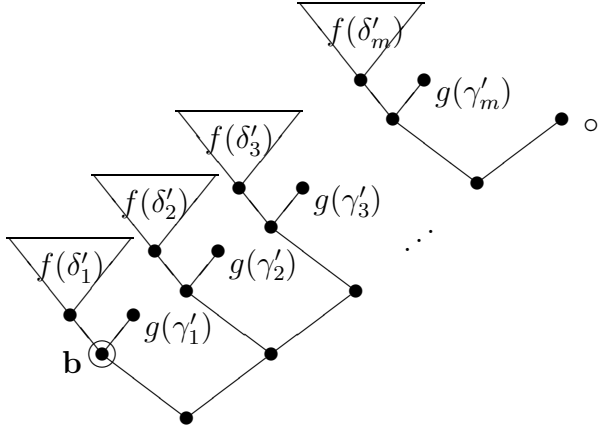
\begin{picture}(1,175)
\put(60,-4){\circle*{5}}
\put(28,20){\circle*{5}}
\put(92,20){\circle*{5}}
\put(60,44){\circle*{5}}
\put(124,44){\circle*{5}}
\put(92,68){\circle*{5}}
\put(170,85){\circle*{5}}
\put(202,109){\circle*{5}}
\put(138,109){\circle*{5}}

\put(150,55){\begin{rotate}{70} $\ddots$ \end{rotate}}

\put(60,-4){\line(4,3){30}}
\put(60,-4){\line(-4,3){30}}
\put(92,20){\line(4,3){30}}
\put(92,20){\line(-4,3){30}}
\put(124,44){\line(-4,3){30}}
\put(170,85){\line(4,3){30}}
\put(170,85){\line(-4,3){30}}

\put(44,27){$g(\gamma'_1)$}
\put(16,35){\circle*{5}}
\put(40,35){\circle*{5}}
\put(28,20){\line(4,5){13}}
\put(28,20){\line(-4,5){13}}

\put(76,51){$g(\gamma'_2)$}
\put(48,59){\circle*{5}}
\put(72,59){\circle*{5}}
\put(60,44){\line(4,5){13}}
\put(60,44){\line(-4,5){13}}

\put(108,75){$g(\gamma'_3)$}
\put(80,83){\circle*{5}}
\put(104,83){\circle*{5}}
\put(92,68){\line(4,5){13}}
\put(92,68){\line(-4,5){13}}

\put(154,116){$g(\gamma'_{m})$}
\put(126,124){\circle*{5}}
\put(150,124){\circle*{5}}
\put(138,109){\line(4,5){13}}
\put(138,109){\line(-4,5){13}}

\put(210,104){$\circ$}

\put(16,35){\line(4,5){23}}
\put(16,35){\line(-4,5){23}}
\put(-8,64){\line(1,0){48}}
\put(3,50){$f(\delta'_1)$}

\put(48,59){\line(4,5){23}}
\put(48,59){\line(-4,5){23}}
\put(24,88){\line(1,0){48}}
\put(35,74){$f(\delta'_2)$}

\put(80,83){\line(4,5){23}}
\put(80,83){\line(-4,5){23}}
\put(56,112){\line(1,0){48}}
\put(67,98){$f(\delta'_3)$}

\put(126,124){\line(4,5){23}}
\put(126,124){\line(-4,5){23}}
\put(102,153){\line(1,0){48}}
\put(113,139){$f(\delta'_{m})$}

\put(13,13){\textbf{b}}
\put(28,20){\circle{10}}

\end{picture}
\ \\
\ \\
There are four different subcases:\\

\textit{i) The root of $f(\delta)$ is mapped into $f(\delta_1')$}\\
Then $f(\delta) \leq f(\delta_1')$, hence, by the induction hypothesis $\delta \leq \delta_1' \leq \delta'$.\\

\textit{ii) The root of $f(\delta)$ is mapped on \textbf{b}}\\
This is only possible if $\delta = \Omega^{\delta_1} \gamma_1$. In this case we have that $f(\delta_1) < f(\delta'_1)$, hence $\delta_1 < \delta_1'$. This implies $\delta< \delta'$.\\

\textit{iii) The root of $f(\delta)$ is mapped into the right immediate subtree of $f(\delta')$}\\
If $m=1$, then $\delta$ has to be $0$, hence $\delta \leq \delta'$. If $m>1$, then $f(\delta) \leq f(\delta')$ yields $f(\delta) \leq f(\Omega^{\delta'_2} \gamma'_2 + \dots + \Omega^{\delta'_m} \gamma'_m) $, hence $\delta < \Omega^{\delta'_1} \gamma'_1 + \dots + \Omega^{\delta'_m} \gamma'_m= \delta'$.\\

\textit{iv) The root of $f(\delta)$ is mapped on the root of $f(\delta')$}\\
Let $\delta= \Omega^{\delta_1}\gamma_1+ \dots + \Omega^{\delta_n}\gamma_n$. 
If the immediate left subtree of $f(\delta)$ is mapped into $f(\delta'_1)$, then like in case 
\textit{ii)}, we obtain that $\delta < \delta'$. So suppose not, then $f(\delta_1) \leq f(\delta'_1)$ and $g(\gamma_1) \leq g(\gamma_1')$. Hence $\delta_1 \leq \delta_1'$ and $\gamma_1 \leq \gamma_1'$. If $\delta_1 < \delta'_1$, then $\delta  < \delta'$. Assume $\delta_1 = \delta'_1$. If $\gamma_1 < \gamma'_1$, then $\delta < \delta'$. Hence assume $\gamma_1 = \gamma'_1$. If $m=1$, then $\delta=\delta'$. So assume $m>1$. There are two cases:

\medskip
\textit{- $n=1$}\\
We obtain easily that $\delta < \delta'$.

\medskip
\textit{- $n \geq 2$}\\
Then $f(\Omega^{\delta_2}\gamma_2 + \dots + \Omega^{\delta_n} \gamma_n) \leq f(\Omega^{\delta'_2}\gamma'_2 + \dots + \Omega^{\delta'_m}\gamma'_m)$. Hence the induction hypothesis yields $\Omega^{\delta_2}\gamma_2 + \dots + \Omega^{\delta_n} \gamma_n \leq \Omega^{\delta'_2}\gamma'_2 + \dots + \Omega^{\delta'_m}\gamma'_m$, hence $\delta \leq \delta'$.
   \end{proof}


\section{Bounds on the proof-theoretical ordinals of $RCA_0 + (\Pi_1^1(\Pi^0_3)\text{--}\,CA_0)^-$ and $RCA_0^* + (\Pi_1^1(\Pi^0_3)\text{--}\,CA_0)^-$}\label{sec:the proof-theoretical ordinals of RCA0 + Pi11CA^- and RCA0* + Pi11CA^-}

In the last section, we show some independence and provability results about the already studied $\wpo$'s. To obtain such results, we need bounds on the proof-theoretical ordinals of theories with light-face $\Pi^1_1$-comprehension, which are obtained in this intermediate section. More specifically, we calculate the proof-theoretical ordinal of $RCA_0 + (\Pi_1^1\text{--}\,CA_0)^-$. As a side question, we were wondering what would happen with the ordinal if we replace $RCA_0$ by $RCA_0^*$. We could not pinpoint down the exact strength of these theories, but we could do it for restricted versions. These restricted theories are fortunately strong enough to obtain the independence and provability results that we want. 

\begin{definition}
Let $n \leq \omega$. 
A $\Pi^1_1(\Pi^0_n)$-formula is a formula of the form $\forall X B(X)$, where $B(X)$ is $\Pi^0_n$. $B$ can contain set and numerical parameters. Note that a $\Pi^0_{\omega}$ formula is the same as an arithmetical formula, hence a $\Pi^1_1(\Pi^0_{\omega})$ is a standard $\Pi^1_1$-formula.
A $(\Pi^1_1(\Pi^0_n))^-$-formula is a formula of the form $\forall X B(X)$, where $B(X)$ is $\Pi^0_n$ and $\forall X B(X)$ contains no free set parameters. It is allowed that $B$ contains numerical parameters. 
\end{definition}

\begin{definition}
Let $(\mathcal{F}\text{--}\,CA_0)$ be the following well-known comprehension scheme
\[\exists Z \forall n (n \in Z \leftrightarrow A(n)),\]
where $A(n)$ is a formula in the class $\mathcal{F}$. If $\mathcal{F}= (\Pi^1_1(\Pi^0_n))^-$, we denote $(\mathcal{F}\text{--}\,CA_0)$ by $(\Pi^1_1(\Pi^0_n)\text{--}\,CA_0)^-$ 
\end{definition}

We show the following results.

\begin{theorem}
\leavevmode
\begin{enumerate}
\item $\vert RCA_0 + (\Pi_1^1(\Pi^0_3)\text{--}\,CA_0)^- \vert = \vartheta(\Omega^\omega)$,
\item $\vert RCA_0^* + (\Pi_1^1(\Pi^0_3)\text{--}\,CA_0)^- \vert \geq \varphi\omega 0 = \vartheta(\Omega \cdot \omega)$.
\end{enumerate}
\end{theorem}

For clarity, we give the definition of $ACA_0$, $RCA_0$ and $RCA_0^*$

\begin{definition}
Both $ACA_0$ and $RCA_0$ are theories in the language of second order arithmetic $L_2$. The axioms of $ACA_0$ are the basic axioms in second order arithmetic and the induction axiom for sets together with comprehension for arithmetical formulas. $RCA_0$ consists of the basic axioms, $\Sigma^0_1$-induction and $\Delta^0_1$-comprehension. Let $L_2(exp)$ be the language $L_2$ augmented by a binary operation symbol $exp$ that denotes the exponential function. Define $RCA_0^*$ as the $L_2(exp)$-theory consisting of the basic axioms, the exponentiation axioms, $\Sigma^0_0$-induction and $\Delta^0_1$-comprehension.
\end{definition}

From more information on reverse mathematics and theories in second order arithmetic, we refer the reader to \cite{sosoa}. 


\subsection{Lower bounds}\label{sec:lower bounds}
These proofs and definitions follow the procedure as in \cite{rathjenweiermann}, but they need some refinements. 
Firstly, we give a primitive recursive ordinal notation system for $\vartheta(\varepsilon_{\Omega+1})$ which is suitable for using in $ACA_0 +(\mathcal{F}\text{--}\,CA_0)$. Then, we introduce an notation system without $\omega$-exponentiation which is suitable to use in $RCA_0 +(\mathcal{F}\text{--}\,CA_0)$ and $RCA_0^*  +(\mathcal{F}\text{--}\,CA_0)$. 

\begin{definition}
Define inductively a set $OT(\vartheta)$ of ordinals and a natural number $G_\vartheta \alpha$ for $\alpha \in OT(\vartheta)$ as follows:
\begin{enumerate}
\item $0 \in OT(\vartheta)$ and $G_\vartheta(0):= 0 $,
\item if $\alpha = \Omega^{\alpha_1} \beta_1 + \dots + \Omega^{\alpha_n} \beta_n$ with $n\geq 1$, $\alpha_1 > \dots > \alpha_n$ and $\Omega > \beta_1,\dots, \beta_n >0$, then
\begin{enumerate}
\item if ($n>1$ or $\alpha_1 >0$) and $\alpha_1,\dots,\alpha_n,\beta_1,\dots,\beta_n \in OT(\vartheta)$, then $\alpha \in OT(\vartheta)$ and $G_\vartheta\alpha :=  \max \{G_\vartheta(\alpha_1), \dots, G_\vartheta(\alpha_n), G_\vartheta(\beta_1), \dots, G_\vartheta(\beta_n)\}+1$,
\item if $n=1$, $\alpha_1=0$ and $\alpha = \omega^{\delta_1}+ \dots + \omega^{\delta_m} > \delta_1 \geq \dots \geq \delta_m$ with $m\geq 2$ and $\delta_1,\dots, \delta_m \in OT(\vartheta)$, then $\alpha \in OT(\vartheta)$ and\\
$G_\vartheta\alpha := \max \{ G_\vartheta(\delta_1), \dots, G_\vartheta(\delta_m)\}+1$,
\end{enumerate}
\item if $\alpha = \vartheta \beta$ and $\beta \in OT(\vartheta)$, then $\alpha \in OT(\vartheta)$ and $G_\vartheta \alpha := G_\vartheta \beta +1$.
\end{enumerate}
\end{definition}

Because $\vartheta\beta$ is always additively closed and $\vartheta$ is injective, the function $G_\vartheta$ is well-defined. Remark that $\Omega \in OT(\vartheta)$, because $\Omega= \Omega^1 \cdot 1$ and $1= \vartheta(0)$.

\begin{lemma}\label{G_theta on OT(theta)}
If $\xi \in OT(\vartheta)$, then $K(\xi) \subseteq OT(\vartheta)$. Furthermore, $G_\vartheta(k(\xi)) \leq G_\vartheta(\xi)$ for all $\xi$ in $OT(\vartheta)$.
\end{lemma}
\begin{proof}
We proof this by induction on $G_\vartheta(\xi)$. If $\xi=0$, then this trivially holds. 
If $\xi = \vartheta(\xi')$, then $K(\xi) = \{\xi\}$, hence this also trivially holds. Assume $\xi = \Omega^{\xi_1} \beta_1 + \dots + \Omega^{\xi_n} \beta_n$ with $n\geq 1$, $\xi_1> \dots > \xi_n$ and $\Omega > \xi_1, \dots, \xi_n > 0$. If $n=1$ and $\xi_1 = 0$, then also $K(\xi) = \{\xi\}$, hence the proof is valid. Let now $n>1$ or $\xi_1 >0$. Then $K(\xi) = \{\beta_1, \dots, \beta_n\} \cup K(\xi_1) \cup \dots \cup K(\xi_n)$. The induction hypothesis yields $K(\xi) \subseteq OT(\vartheta)$. Furthermore, $G_\vartheta(k(\xi_i)) \leq G_\vartheta(\xi_i) < G_\vartheta(\xi)$. Therefore, the strict inequality $G_\vartheta(k(\xi)) = G_\vartheta(\max_i\{k(\xi_i),\beta_i\}) < G_\vartheta(\xi)$ holds.
\end{proof}

Each ordinal $\alpha \in OT(\vartheta)$ has a unique normal form using the symbols $0, \omega, \Omega,+,$ $ \vartheta$. The relation $\alpha < \beta$ can expressed using the ordinals appearing in their normal form (by Lemma \ref{main property theta-function}), which have strictly smaller $G_\vartheta$-values by the previous Lemma \ref{G_theta on OT(theta)}. Hence, the following lemma holds.

\begin{lemma}\label{ONS for Howard-Bachmann}
If we use a specific coding of $(OT(\vartheta),<_{OT(\vartheta)})$ in the natural numbers, then $(OT(\vartheta) \cap \Omega ,<_{OT(\vartheta)})$ can be interpreted as a primitive recursive ordinal notation system for the ordinal $\vartheta(\varepsilon_{\Omega+1})$. Furthermore, one can choose this coding in such a way that $\forall \xi \in K(\alpha) (\xi \leq_{\mathbb{N}} \alpha)$.
\end{lemma}

This is the ordinal notation system that we will use if we work in the theory $ACA_0 +(\mathcal{F}\text{--}\,CA_0)$. However, in the theory $RCA_0 + (\mathcal{F}\text{--}\,CA_0)$ and $RCA_0^* + (\mathcal{F}\text{--}\,CA_0)$, we use a different ordinal notation system $OT'(\vartheta)$ without $\omega$-exponen\-ti\-a\-tion.

\begin{definition}
Define inductively a set $OT'(\vartheta)$ of ordinals and a natural number $G'_\vartheta \alpha$ for $\alpha \in OT'(\vartheta)$ as follows:
\begin{enumerate}
\item $0\in OT'(\vartheta)$ and $G'_\vartheta(0):=0$,
\item if $\alpha = \Omega^{n} \beta_n + \dots + \Omega^{0} \beta_0$, with $\Omega > \beta_n > 0$ and $\Omega > \beta_0,\dots, \beta_{n-1}$, then
\begin{enumerate}
\item if $n>0$ and $\beta_0,\dots,\beta_n \in OT'(\vartheta)$, then $\alpha \in OT'(\vartheta)$ and \\
$G'_\vartheta\alpha := \max \{G'_\vartheta(\beta_0), \dots , G'_\vartheta(\beta_n)\}+1$,
\item if $n=0$ and $\alpha =_{NF} \delta_1+ \dots + \delta_m$ with $m\geq 2$ and $\delta_1,\dots, \delta_m \in OT'(\vartheta)$, then $\alpha \in OT'(\vartheta)$ and $G'_\vartheta\alpha := \max \{G'_\vartheta(\delta_1), \dots , G'_\vartheta(\delta_m)\}+1$,
\end{enumerate}
\item if $\alpha = \vartheta \beta$ and $\beta \in OT'(\vartheta)$, then $\alpha \in OT'(\vartheta)$ and $G'_\vartheta \alpha := G'_\vartheta \beta +1$.
\end{enumerate}
$NF$ stands for normal form.
\end{definition}

Also in this ordinal notation system, $G'_\vartheta$ is well-defined.

\begin{lemma}\label{G'_theta on OT'(theta)}
If $\xi \in OT'(\vartheta)$, then $K(\xi) \subseteq OT'(\vartheta)$. Furthermore, $G'_\vartheta(k(\xi)) \leq G'_\vartheta(\xi)$ for all $\xi$ in $OT'(\vartheta)$.
\end{lemma}
\begin{proof}
If $\xi=0$, then this trivially holds. 
If $\xi = \vartheta(\xi')$, then $K(\xi) = \{\xi\}$, hence this also trivially holds. 
Assume $\xi = \Omega^{n} \xi_n + \dots + \Omega^{0} \xi_0$ with $\Omega>\xi_n > 0$ and $\Omega > \xi_1, \dots, \xi_{-1}$. If $n=0$, then also $K(\xi) = \{\xi\}$, hence the proof is valid. Let now $n>0$. Then $K(\xi) \subseteq \{n,\dots,0,\xi_0, \dots, \xi_n\}$. We know $\xi_0,\dots,\xi_n \in OT'(\vartheta)$. Additionally, it is trivial to show $\{n,\dots,0\} \subseteq OT'(\vartheta)$. Hence, $K(\xi) \subseteq OT'(\vartheta)$. 
It is also trivial to show that $G'_\vartheta(m) \leq 2$ for all natural numbers $m$. Therefore, $G'_\vartheta(n) \leq G'_\vartheta(\xi)$ because $G'_\vartheta(\xi_n)\geq 1$. Also, $G'_\vartheta(\xi_i) < G'_\vartheta(\xi)$. Hence $G'_\vartheta(k(\xi)) = G'_\vartheta(\max_i\{n,\xi_i\}) < G'_\vartheta(\xi)$.
\end{proof}

Like for the first ordinal notation system, we have to following lemma.

\begin{lemma}\label{ONS for ackermann}
If we use a specific coding of $(OT'(\vartheta),<_{OT'(\vartheta)})$ in the natural numbers, $(OT'(\vartheta) \cap \Omega,<_{OT'(\vartheta)})$ can be interpreted as a primitive recursive ordinal notation system for the ordinal $\vartheta(\Omega^\omega)$. Furthermore, one can choose this coding in such a way that $\forall \xi \in K(\alpha) (\xi \leq_{\mathbb{N}} \alpha)$.
\end{lemma}

From now on, we fix primitive recursive ordinal notation systems $OT(\vartheta)\cap \Omega$ and $OT'(\vartheta)\cap \Omega$. If we mention $ACA_0$ in the beginning of a theorem, then we assume that we work in $OT(\vartheta)$. Similarly, if we mention $RCA_0$ or $RCA_0^*$ in the beginning of a theorem, we assume that we work in $OT'(\vartheta)$.

\begin{definition}
\begin{align*}
Prog(\prec,F)&:= \forall x (\forall y(y\prec x \rightarrow F(y)) \rightarrow F(x)),\\
TI(\prec,F)& := Prog(\prec,F) \rightarrow \forall x F(x),\\
WF(\prec) & := \forall X TI(\prec,X),
\end{align*}
where the formula $Prog(\prec,F)$ stands for progressiveness, $TI(\prec,F)$ for transfinite induction and $WF(\prec)$ for well-foundedness. $F(x)$ is an arbitrary $L_2(exp)$-formula if we work in $RCA_0^*$ or an arbitrary $L_2$-formula if we work in $RCA_0$ or $ACA_0$. For an element $\alpha \in OT(\vartheta)$, the formula $WF(\alpha)$ stands for `$<_{OT(\vartheta)}$ restricted to $\{\beta \in OT(\vartheta) : \beta < \alpha \}$ is well-founded'. We sometimes also denote this as $WF(< \upharpoonright \alpha)$. Similarly for $OT'(\vartheta)$.
\end{definition}

\begin{definition}
\begin{enumerate}
\item $Acc:= \{\alpha < \Omega : WF(< \upharpoonright \alpha)\}$,
\item $M := \{\alpha : K(\alpha) \subseteq Acc\}$,
\item $\alpha  <_{\Omega} \beta \Leftrightarrow \alpha, \beta \in M \wedge \alpha < \beta$.
\end{enumerate}
\end{definition}

The next lemma shows that $Acc$, $M$ and $<_{\Omega}$ can be expressed by a $(\Pi^1_1(\Pi^0_3))^-$-formula.

\begin{lemma}
$Acc$, $M$ and $<_{\Omega}$ are expressible by a $(\Pi^1_1(\Pi^0_3))^-$-formula.
\end{lemma}
\begin{proof}
The proof is the same for the ordinal notation system $OT(\vartheta)$ and $OT'(\vartheta)$.
\[WF(\alpha) = \forall X \left( \forall x (\forall y(y \prec x \rightarrow y \in X) \rightarrow x \in X) \rightarrow \forall x (x \in X) \right),\]
where  $\prec$ is $< \upharpoonright \alpha$. It is easy to see that the prenex normal form of the formula $WF(\alpha)$ is $(\Pi^1_1(\Pi^0_3))^-$, hence $Acc$ can be expressed by such a formula.
$M$ can be represented by the formula $\forall \xi \leq_{\mathbb{N}} \alpha ( \xi \in K(\alpha) \rightarrow \xi \in Acc)$. Because $\xi \in K(\alpha)$ is elementary recursive, both $M$ and $<_\Omega$ are also expressible by a $(\Pi^1_1(\Pi^0_3))^-$-formula.
 
\end{proof}


\begin{lemma}
\label{Proof of + in RCA_0^*}
$(RCA_0^*)$ $\alpha$, $\beta \in Acc \Longrightarrow \alpha + \beta \in Acc$.
\end{lemma}
\begin{proof}
Obvious.
   \end{proof}





\begin{lemma}
\label{Proof of exponentiation in ACA_0}
$(ACA_0)$ $\alpha \in Acc \Longrightarrow \omega^\alpha \in Acc$. 
\end{lemma}
\begin{proof}
A proof of this lemma goes back to Gentzen. See \cite{Pohlersboek1,Prooftheoryschutte}.
\end{proof}

\begin{definition}
Let $Prog_{\Omega}(X)$ be the formula
\[(\forall \alpha \in M)\left[ (\forall \beta <_{\Omega} \alpha) (\beta \in X ) \rightarrow \alpha \in X \right].\]
Let $Acc_{\Omega}$ be the set $\{\alpha \in M : \vartheta(\alpha) \in Acc\}$.
\end{definition}


\begin{lemma}\label{Prog in RCA_0*} $(RCA_0^*)$
$Prog_{\Omega}(Acc_{\Omega})$. 
\end{lemma}
\begin{proof}
We work in $OT'(\vartheta)$.
Assume $\alpha \in M$ and $(\forall \beta<_\Omega \alpha)( \vartheta(\beta) \in Acc)$. We want to proof that $\vartheta(\alpha) \in Acc$. We show that $(\forall \xi < \vartheta(\alpha)) (\xi \in Acc)$ by induction on $G'_\vartheta \xi$, from which the theorem follows. If $\xi=0$, then this trivially holds. So assume $\xi > 0$.
\medskip

\textit{a) Assume $\xi \notin P$}\\
Because $\xi< \Omega$, we have $\xi=_{NF} \xi_1 + \dots  + \xi_n > \xi_1 \geq \dots \geq \xi_n$ ($n\geq 2$). The induction hypothesis yields $\xi_i \in Acc$. Hence from Lemma \ref{Proof of + in RCA_0^*}, we obtain $\xi \in Acc$.
\medskip

\textit{b) Assume $\xi = \vartheta(\xi')$}\\
From $\vartheta(\xi') < \vartheta(\alpha)$, we obtain either $\xi' < \alpha$ and $k(\xi') < \vartheta(\alpha)$ or $\vartheta(\xi') \leq k(\alpha)$. 
In the former case, $G'_\vartheta(k(\xi')) \leq G'_\vartheta(\xi') < G'_\vartheta(\xi)$ and the induction hypothesis implies $k(\xi') \in Acc$, hence $K(\xi') \subseteq Acc$. So $\xi' \in M$, from which it follows $\xi' <_\Omega \alpha$, hence $\vartheta(\xi') \in Acc$. 
In the latter case, we know that $k(\alpha) \in Acc$, because $\alpha \in M$. Therefore, $\vartheta(\xi') \in Acc$.
   \end{proof}

We actually do not need the following lemma, because we already know the proof-theoretical ordinal of $ACA_0 + (\Pi^1_1$--\,$CA_0)^-$. However, just for  the completeness, we mention it here.
\begin{lemma}\label{Prog in ACA_0} $(ACA_0)$
$Prog_{\Omega}(Acc_{\Omega})$. 
\end{lemma}
\begin{proof}
The proof uses $OT(\vartheta)$ and follows the same procedure as Lemma \ref{Prog in RCA_0*}. The only difference is the usage of Lemma \ref{Proof of exponentiation in ACA_0}.
   \end{proof}

\begin{lemma}\label{Prog implies WF in RCA_0}
Let $A(a)$ be a $(\Pi^1_1(\Pi^0_3))^-$-formula. Define $A_k$ as 
\[\forall \alpha [(\forall \beta <_\Omega \alpha) A(\beta) \rightarrow (\forall \beta <_\Omega \alpha + \Omega^k) A(\beta)].\]
Then $RCA_0 + (\Pi^1_1(\Pi^0_3)$-- $CA_0)^-$ proves $Prog_\Omega(\{\xi: A(\xi)\}) \rightarrow A_k$.
\end{lemma}
\begin{proof}
We will prove that by outer induction on $k$. First note that $A(a)$ can be expressed by a set in $RCA_0 + (\Pi^1_1(\Pi^0_3)$-- $CA_0)^-$. Assume that $Prog_\Omega(\{\xi: A(\xi)\})$ and pick an arbitrary $\alpha$ such that $(\forall \beta <_\Omega \alpha) A(\beta)$. If $\alpha \notin M$, the assertion trivially follows. Assume $\alpha \in M$. 
If $k=0$, the proofs follows easily from $Prog_\Omega(\{\xi: A(\xi)\})$. Assume $k=l+1$ en suppose $(\forall \beta <_\Omega \alpha) A(\beta) $. We want to prove that $(\forall \beta <_\Omega \alpha + \Omega^{l+1}) A(\beta)$. Take an arbitrary $\beta <_\Omega \alpha + \Omega^{l+1}$. $RCA_0$ proves that there exists a $\xi \in Acc$ such that $\beta <_\Omega \alpha + \Omega^{l}\xi$ (by induction on the construction of $\beta$ in $OT(\vartheta)$, one can show that one can take $\xi = k(\beta) +1$. Let $B(\zeta)$ be
\[(\forall \beta <_\Omega \alpha + \Omega^l \zeta) A(\beta).\]
$B(\zeta)$ is a $\Pi^0_1$-formula in $A$. It is known (Lemma 6 in \cite{RMandrankfunctionsfordirectedgraphs}) that $RCA_0 \vdash \forall X (WO(X) \rightarrow TI_X(\psi))$ for all $\Pi^0_1$-formulas $\psi$. Because $\xi \in Acc$, we have that $\xi+1$ is well-ordered, hence we know that $TI_{\xi+1}(B)$ is true. This means
\[\forall x \leq \xi [\forall y \leq \xi (y<x \rightarrow B(y) ) \rightarrow B(x)] \rightarrow \forall x \leq \xi B(x).\]
The theorem follows from $B(\xi)$, hence we only have to prove the progressiveness of $B$ along $\xi+1$. 

\medskip

Assume that $x \leq \xi$. Then $x \in Acc$. If $x=0$, then $B(0)$ follows from the assumption. Assume that $x$ is a limit. If $\beta <_\Omega \alpha + \Omega^l x $, then there exists a $y < x$ such that $\beta <_\Omega \alpha + \Omega^l y$. Because $B(y)$ is valid, one obtains $A(\beta)$.
Assume that $x = x'+1 \in Acc$. From $x' < x$, one obtains $(\forall \beta <_\Omega \alpha + \Omega^l x') A(\beta)$. Because $A_l$ is valid, we get $(\forall \beta <_\Omega \alpha + \Omega^l x' + \Omega^l) A(\beta)$, hence $(\forall \beta <_\Omega \alpha + \Omega^l (x' +1)) A(\beta)= B(x)$.

   \end{proof}


\begin{theorem}\label{lower bound RCA_0 + Pi11-CA}
$\vert RCA_0 + (\Pi^1_1(\Pi^0_3)$-- $CA_0)^- \vert \geq \vartheta(\Omega^\omega)$.
\end{theorem}
\begin{proof}
This follows from Lemmas \ref{Prog in RCA_0*}, \ref{Prog implies WF in RCA_0} and the fact that $Acc_\Omega$ is expressible by a $(\Pi^1_1(\Pi^0_3))^-$-formula.
   \end{proof}


\begin{lemma}\label{Prog implies WF in RCA_0^*}
Let $A(a)$ be a $(\Pi^1_1(\Pi^0_3))^-$-formula. Define $A_k$ as 
\[(\forall \beta <_\Omega \Omega  \cdot k) A(\beta).\]
Then $RCA_0^* + (\Pi^1_1(\Pi^0_3)$-- $CA_0)^-$ proves $Prog_\Omega(\{\xi: A(\xi)\}) \rightarrow A_k$.
\end{lemma}
\begin{proof}
We will prove that by outer induction on $k$. First note that $A(a)$ can be expressed by a set in $RCA_0^* + (\Pi^1_1(\Pi^0_3)$-- $CA_0)^-$. It is easy to see that the case $k=0$ holds. Assume $k=l+1$ and $Prog_\Omega(\{\xi: A(\xi)\})$. Then we know $(\forall \beta <_\Omega \Omega  \cdot l) A(\beta)$. Pick an arbitrary $\beta <_\Omega \Omega \cdot (l+1)$. If $\beta < \Omega \cdot l$, we obtain $A(\beta)$. Hence assume that $\beta \geq \Omega \cdot l$. There e\-xists a $\xi < \Omega$ such that $\beta = \Omega \cdot l + \xi$. Let $B(\zeta)$ be $A( \Omega \cdot l + \zeta)$. $B(\zeta)$ is a $\Pi^0_0$-formula in $A$. We prove by induction on $\zeta$ that $(\forall \zeta \in Acc) B(\zeta)$ is true. From this, the theorem follows.\\
If $\zeta=0$, then $B(\zeta)$ is true because $(\forall \beta <_\Omega \Omega  \cdot l) A(\beta)$ and $Prog_\Omega(\{\xi: A(\xi)\})$ imply $A( \Omega \cdot l )$. Assume $\zeta \in Acc$ is a limit and assume $(\forall \zeta' <_\Omega \zeta) B(\zeta)$. From $Prog_\Omega(\{\xi: A(\xi)\})$, we obtain $B(\zeta)$. Let $\zeta = \zeta'+1 \in Acc$. Then $B(\zeta)$ follows from $B(\zeta')$ and $Prog_\Omega(\{\xi: A(\xi)\})$.
%
%
%
   \end{proof}

\begin{theorem}\label{lower bound RCA_0^* + Pi11-CA}
$\vert RCA_0^* + (\Pi^1_1(\Pi^0_3)$--\,$CA_0)^- \vert \geq \vartheta(\Omega \cdot \omega) = \varphi \omega 0$.
\end{theorem}
\begin{proof}
This follows from Lemmas \ref{Prog in RCA_0*}, \ref{Prog implies WF in RCA_0^*} and the fact that $Acc_\Omega$ is expressible by a $(\Pi^1_1(\Pi^0_3))^-$-formula.
   \end{proof}


\subsection{Upper bounds}\label{sec:upper bounds}
In this subsection, we give an upper bound for $\vert RCA_0 + (\Pi^1_1(\Pi^0_3)\text{--}\,CA_0)^- \vert $. For this, we use the fact that $\vert \Pi^1_2$--$BI_0 \vert  = \vartheta(\Omega^\omega)$ (see \cite{rathjenweiermann}).

\begin{lemma}\label{special normal form}
For every arithmetical formula $B(X)$ with all free set variables indicated, there is a $\Delta_0$-formula $R(x,X,f)$ such that
\begin{enumerate}
\item $ACA_0 \vdash B(X) \rightarrow \exists f \forall x R(x,X,f)$
\item If $T$ is a theory with $RCA_0 \subseteq T$ and $\mathcal{F}(x,y)$ is an arbitrary formula,
\[T \vdash \forall x \exists ! y \mathcal{F}(x,y) \wedge \forall x \exists z (lh(z)=x \wedge \forall i < x \mathcal{F}(i,(z)_i)),\]
then $T \vdash \forall x R(x,X,\mathcal{F}) \rightarrow B(X)$, where $R(x,X,\mathcal{F})$ results from $R(x,X,f)$ by replacing subformulae of the form $f(t)=s$ by $\mathcal{F}(t,s)$.
\end{enumerate}
\end{lemma}
\begin{proof}
This proof is a little adaptation of the normal form theorem V.1.4 in \cite{sosoa}. We can assume that $B(X)$ is in prenex normal form
\[B(X) \equiv \forall x_1 \exists y_1 \dots \forall x_r \exists y_r S(x_1,y_1,\dots,x_r,y_r,X),\]
with $S$ quantifier-free. Then over $ACA_0$, we have
\begin{align*}
B(X) & \leftrightarrow \exists f_1 \forall x_1 x_2 \exists y_2 \dots S(x_1,f_1(x_1), x_2,y_2,\dots,x_r,y_r,X)\\
& \leftrightarrow \exists f_1 \exists f_2 \forall x_1 x_2 x_3 \exists y_3 \dots S(x_1,f_1(x_1),x_2,f_2(x_1,x_2),x_3,y_3,\dots,x_r,y_r,X)\\
& \dots\\
& \leftrightarrow \exists f_1 \dots f_r \forall x_1 \dots x_r\, S(x_1,f_1(x_1),x_2,f_2(x_1,x_2),\dots, x_r,f_r(x_1,\dots, x_r),X)\\
& \leftrightarrow \exists f \forall x\, S((x)_1,(f)_1((x)_1),\dots,(x)_r, (f)_r((x)_1,\dots,(x)_r),X),
\end{align*}
where $(f)_k((x)_1,\dots,(x)_k) = f(\langle  k,(x)_1,\dots,(x)_k \rangle)$ and $\langle \dots \rangle$ is some primitive recursive tupel coding. Now let $R(x,X,f) := S((x)_1, (f)_1((x)_1),\dots,X)$. Note that the right-to-left directions are provable in $RCA_0$, so that part 2. follows from reading the equivalences from bottom to top.
   \end{proof}

By adapting the previous lemma, we get the following corollary.

\begin{corollary}\label{corollary special normal form}
If in addition to the conditions of Lemma \ref{special normal form}(2.), $T$ also satisfies
\[T \vdash \forall x \exists! y \mathcal{G}(x,y) \wedge \forall x \exists z  (lh(z)=x \wedge \forall i < x \mathcal{G}(i,(z)_i)),\]
then $T \vdash \forall x R(x,\mathcal{G},\mathcal{F}) \rightarrow B(\mathcal{G})$, where $R(x,\mathcal{G},\mathcal{F})$ results from $R(x,X,\mathcal{F})$ by replacing $t\in X$ by $\mathcal{G}(t,0)$.
\end{corollary}

\begin{lemma}\label{special normal form 2}
Assuming that the conditions for $\mathcal{F}$ and $\mathcal{G}$ from Lemma \ref{special normal form} and Corollary \ref{corollary special normal form} are satisfied, then there exists a $\Delta^0_1$-formula $P$ such that
\[T \vdash  \forall x R(x,\mathcal{G},\mathcal{F}) \leftrightarrow \forall x P(\mathcal{G}[x], \mathcal{F}[x])\]
\end{lemma}
\begin{proof}
By induction on the build-up of $R$.
   \end{proof}

\begin{theorem}\label{upper bound RCA_0 + Pi11(Pi03)-CA-}
Assume that $B(X,n)$ is a $\Pi^0_3$-formula, then
\[\Pi^1_2\text{--}\,BI_0 \vdash \forall n [ \forall A \in Rec((\Pi^1_1)^-) B(A(\cdot),n) \leftrightarrow (\forall X \subseteq \omega) B(X,n) ],\]
where $B(A(\cdot),n)$ results from $B(X,n)$ by replacing $t \in X$ by $A(t)$.
\end{theorem}
\begin{proof}
Pick an arbitrary natural number $n_0$. Assume that we have \[(\exists A \in Rec((\Pi^1_1)^-) ) \text{ } \neg B(A(\cdot),n_0).\] In \cite{barinductionandomegamodelreflection} it is proven that $\Sigma^1_4$-$\omega$-model reflection holds in $\Pi^1_2$--$BI_0$. Therefore, there exists an $\omega$-model $M$ such that $n_0 \in M$ and $M \models ACA_0 + \neg B(A,n_0)$. Define $X$ as the set $\{ m \in \omega : M \models A(m)\}$. This set exists by arithmetical comprehension. Hence, $\neg B(X,n_0)$ is valid.
\medskip

For the reverse implication, assume $(\exists X \subseteq \omega) \neg B(X,n_0)$. Using Lemmas \ref{special normal form} and \ref{special normal form 2} on the formula $\neg B(X,n_0)$, we have $\exists X \exists f \forall x P(X[x], f[x])$, thus $\exists h \forall x P(h_0[x],h_1[x])$. Define $I(\sigma)$ as $\exists h (\sigma \subseteq h \wedge \forall x P(h_0[x],h_1[x]))$. 
Define a class function $\mathcal{H}$ by $\mathcal{H}(0) = \langle \rangle$ and $\mathcal{H}(n+1) = \mathcal{H}(n)^{\frown}\langle \mu k \cdot I(\mathcal{H}(n)^{\frown} \langle k \rangle) \rangle$. Note that $\mathcal{H}$ has a $\Sigma^1_2$--graph. As $\Pi^1_2$--$BI_0$ proves $\Sigma^1_2$-induction on the naturals, one can prove that $\mathcal{H}(n)$ is always defined by $\exists h \forall x P(h_0[x],h_1[x])$. Also note that $\mathcal{H}$ is recursive in $I$. We can show that $I$ is recursive in a $(\Pi^1_1)^-$-formula. This because $I$ is recursive in $W$, the set of all recursive well-founded trees, which is $(\Pi^1_1)^-$-complete. Therefore, we obtain that $\mathcal{H} \in Rec((\Pi^1_1)^-)$. 
We also have $\forall x P(\mathcal{H}_0[x], \mathcal{H}_1[x])$ by this construction. Letting $\mathcal{G}= \mathcal{H}_0$ and $\mathcal{F} = \mathcal{H}_1$ in Lemma \ref{special normal form 2}, we get $\forall x R(x,\mathcal{G},\mathcal{F})$. By Corollary \ref{corollary special normal form} we get $\neg B(\mathcal{G},n_0)$.
   \end{proof}

\begin{corollary}
$\vert RCA_0 + (\Pi^1_1(\Pi^0_3)\text{--}\,CA_0)^- \vert  = \vartheta(\Omega^\omega)$. 
\end{corollary}
\begin{proof}
Follows from Theorems \ref{lower bound RCA_0 + Pi11-CA}, \ref{upper bound RCA_0 + Pi11(Pi03)-CA-} and the fact that $ \vert \Pi^1_2\text{--}\,BI_0 \vert = \vartheta(\Omega^\omega)$.
 
\end{proof}

In a similar way, one could also show the following lemma.

\begin{lemma}
Assume that $B(X,n)$ is a $\Pi^0_2$-formula, then
\[\Pi^1_1\text{--}\,BI_0 \vdash \forall n [ \forall A \in Rec((\Pi^1_1)^-) B(A(\cdot),n) \leftrightarrow (\forall X \subseteq \omega) B(X,n) ].\]
\end{lemma}

Therefore, $\vert \Pi^1_1\text{--}\,BI_0 \vert = \varphi\omega 0$ yields that the ordinal of $RCA_0^* + (\Pi^1_1(\Pi^0_2)\text{--}\,CA_0)^-$ is bounded above by $\varphi \omega 0$.
Actually one can prove that the ordinal of the theory $RCA_0^* + (\Pi^1_1(\Pi^0_2)\text{--}\,CA_0)^-$ is even much lower. Using $WKL_0$ and the normal form theorem II.2.7 in \cite{sosoa}, one can prove that every $(\Pi^1_1(\Pi^0_2))^-$-formula is equivalent with a $(\Pi^0_2)^-$-formula (and if the original formula has extra parameters, then the equivalent one will also have those parameters).
The intuitive idea behind this is that $WKL_0$ proves that the projection of closed set is a closed set, meaning that $\exists X \forall x\dots $ can be reduced to $\forall z \dots$ (a closed set can be seen in some sense as a $\Pi^0_1$-formula). 
Furthermore, one can proof the following lemmas.

\begin{lemma}
Let $\mathcal{F}$ be $(\Pi^1_1(\Pi^0_n))^-$ or $(\Pi^0_n)^-$. Then the first order part of the theory $WKL_0 + \mathcal{F}$--$\,CA_0$ is the same as that of $RCA_0 + \mathcal{F}$--$\,CA_0$.
\end{lemma}
\begin{proof}
Can be proved easily by little adaptations of paragraph IX.2 in \cite{sosoa}.
 
\end{proof}
In a similar way, 
\begin{lemma}
The first order part of $RCA_0 + (\Pi^0_2)^-$--IND is $I\Sigma_3$. 
\end{lemma}
\begin{proof}
Can be proved by adaptations of paragraph IX.1 in \cite{sosoa}.
 
\end{proof}

Hence, $\vert RCA_0 + (\Pi^1_1(\Pi^0_2)$--$\,CA_0)^-\vert  = \vert I\Sigma_3 \vert = \omega^{\omega^{\omega^\omega}}$. In a similar but more technical way, one could prove that $RCA^*_0 + (\Pi^1_1(\Pi^0_2))^- $ is $\Pi_4$-conservative over $I\Sigma_2$. Hence, $\vert RCA^*_0 + (\Pi^1_1(\Pi^0_2))^- \vert = \omega^{\omega^\omega}$. We want to thank Leszek Ko\l odziejczyk, who reminded us of these facts. 

\medskip

There are still some open questions left, for example what is the ordinal of $RCA_0 + (\Pi^1_1(\Pi^0_4)\text{--}\,CA_0)^-$ or $RCA_0 + (\Pi^1_1\text{--}\,CA_0)^-$ etc. We state the following conjectures

\begin{conjecture}
\leavevmode
\begin{enumerate}
\item $\vert RCA_0^* + (\Pi^1_1(\Pi^0_3)\text{--}\,CA_0)^- \vert = \varphi \omega 0$,
\item $\vert RCA_0^* + (\Pi^1_1\text{--}\,CA_0)^- \vert = \varphi \omega 0$,
\item $\vert RCA_0 + (\Pi^1_1\text{--}\,CA_0)^- \vert = \vartheta(\Omega^\omega)$.
\end{enumerate}
\end{conjecture}


\section{Independence results}\label{sec:independence results}

\subsection{Independence results for $ACA_0 +  (\Pi^1_1$--\,$CA_0)^{-} $}\label{sec:independence results for ACA_0 +  (Pi^1_1--CA_0)^{-}}

%
%

It is well-known that $\vert  ACA_0 + (\Pi^1_1$--\,$CA_0)^{-}  \vert = \vartheta(\varepsilon_{\Omega+1})$ (for example see \cite{897}).
Now,
\begin{theorem}\label{InACA0: T(B) wpo ->WF(Howard-Bachmann)} $ACA_0 \vdash$ `$\T(\mathcal{B}(\cdot))$ is a $\wpo$' $\rightarrow WF(\vartheta(\varepsilon_{\Omega+1}))$.
\end{theorem}
\begin{proof}
Theorem \ref{computation lower bound W(X)=B(X)} can be carried out in $ACA_0$ if we use the primitive recursive ordinal notation system $OT(\vartheta) \cap \Omega$ for $\vartheta(\varepsilon_{\Omega+1})$.
   \end{proof}

Therefore, one has the following theorem,
\begin{theorem} $ACA_0 + (\Pi^1_1$--\,$CA_0)^{-} \not\vdash$ `$\T(\mathcal{B}(\cdot))$ is a $\wpo$'.
\end{theorem}

The $\wpo$ $\T(\mathcal{B}(\cdot))$ can be seen as the limit of $\T(\cdot^{\overbrace{*\dots*}^{n}})$, because a binary tree can be seen as an iteration of the ${}^*$-operator. For example, one can interpret an element of ${\{a\}}^{**}$ as a binary tree which goes only one time to the left and every node has label $a$.
So, it would be interesting if $ACA_0 + (\Pi^1_1$--\,$CA_0)^{-} \vdash `\T(\cdot^{\overbrace{*\dots*}^{n}})$ is a $\wpo$' for every natural number $n$ because this theory does not prove the `limit'.\\
The maximal order type of $\T(\cdot^{*\dots*})$ is strictly below the proof-theoretical ordinal of $ACA_0 + (\Pi^1_1$--\,$CA_0)^- $. Therefore, one can really expect this provability assertion. However, in the proof of the upper bound of the maximal order type of $\T(\cdot^{*\dots*})$ one uses several induction schemes and it is not immediately clear that the proof goes through in $ACA_0 + (\Pi^1_1$--\,$CA_0)^- $. Theorem \ref{provability in Pi11CA0 and ACA} shows that this is indeed possible and uses the so-called minimal bad sequence argument. For more information about the minimal bad sequence argument and its reverse mathematical strength, we refer the reader to \cite{onthelogicalstrengthofnashwilliamstheoremontransfinitesequences}.

\begin{lemma}\label{star} Over $RCA_0$, the following are equivalent
\begin{enumerate}
\item $ACA_0$,
\item Higman's theorem, i.e. $\forall X(X$ is a $\wpo$ $\rightarrow$ $X^*$ is a $\wpo)$.
\item $\forall X(X$ is a well-quasi-order $\rightarrow$ $X^*$ is a well-quasi-order$)$.
\end{enumerate}
\end{lemma}
\begin{proof}
It is trivial to show that (2.) and (3.) are equivalent. In \cite{girard87} and \cite{reversemathematicsandordinalexponentiation}, the reader finds a proof of the fact that $ACA_0$ is equivalent over $RCA_0$ with $\forall \alpha(\alpha$ is a well-order $\rightarrow$ $2^\alpha$ is a well-order$)$. One can prove that the latter is equivalent with Higman's theorem. For a detailed version see \cite{clote90}. 
 
\end{proof}

\begin{theorem}\label{provability in Pi11CA0 and ACA} For all $n$, $ACA_0 + (\Pi^1_1$--\,$CA_0)^-\vdash$ `$\T(\cdot^{\overbrace{*\dots*}^{n}})$ is a $\wpo$'.
\end{theorem}
\begin{proof} Fix a natural number $n$. We reason in $ACA_0 + (\Pi^1_1$--\,$CA_0)^{-}$. Code the elements of $\T(\cdot^{\overbrace{*\dots*}^{n}})$ as natural numbers such that $\circ$ has code $0$ and the code of $t_i$ is strictly smaller than the code of $t = \circ[w(t_1,\dots,t_i,\dots, t_n)]$. This coding can be done primitive recursively. We say that the leaves of $\circ[w(t_1,\dots,t_i,\dots, t_n)]$ are $\{t_1,\dots,t_n\}$ and we assume that there is a primitive recursive relation `\dots is a leaf of \dots'. 
If $\sigma$ decodes a finite sequence, then this sequence is equal to $((\sigma)_0,\dots,(\sigma)_{lh(\sigma)-1})$, where $lh(\sigma)$ is the length and $(\sigma)_i$ is the $i$-th element of the sequence. An infinite sequence $(\sigma_i)_{i< \omega}$ is decoded as a set $\{(i,\sigma_i): i< \omega\}$. $\left( \{(i,\sigma_i): i< \omega\} \right)_i$ stands for the element $\sigma_i$. Note that one can recursively construct the set $\{(\sigma_0,\dots, \sigma_i) : i<\omega\}$ from the original set $\{(i,\sigma_i): i< \omega\}$, where $(\sigma_0,\dots, \sigma_i) $ is decoded as a natural number. Furthermore, if one has the set $\{(\sigma_0,\dots, \sigma_i) : i<\omega\}$, one can reconstruct the original set $\{(i,\sigma_i): i< \omega\}$ from it in a recursive way.

\medskip 

Now, assume that $\T(\cdot^{*\dots *})$ is not a well-partial-order. Then there exists an infinite sequence $(t_i)_{i<\omega}$ in $\T(\cdot^{*\dots *})$ such that $\forall i,j (i<j \rightarrow t_i \not\leq t_j)$. Define $\chi(\sigma)$ as
\begin{align*}
{}&{}\text{$\sigma$ is a finite sequence of elements in $\T(\cdot^{*\dots *})$}\\
 \text{and }{}&{} \exists Z( \text{$Z$ is an infinite bad sequence in $\T(\cdot^{*\dots *})$} \wedge \forall i < lh(\sigma) ((\sigma)_i=(Z)_i)),
\end{align*}
and $\psi(\sigma)$ as  
\begin{align*}
&\text{$\sigma$ is a finite sequence of elements in $\T(\cdot^{*\dots *})$}\\
 \text{and }& \forall Y \left[\left( \text{$Y$ is an infinite bad sequence in $\T(\cdot^{*\dots *})$}\right) \right.\\
& \left. \rightarrow \forall i < lh(\sigma) \left[\forall k < i  \left((\sigma)_k =(Y)_k\right)\rightarrow (\sigma)_i \leq (Y)_i\right]\right],
\end{align*}
Note that $(\sigma)_i \leq (Y)_i$ is interpreted as the inequality relation between natural numbers and not between elements of $\T(\cdot^{*\dots* })$. 
Using $(\Pi^1_1$--\,$CA_0)^-$, there exists a set $S$ such that $\sigma \in S \leftrightarrow \chi(\sigma) \wedge \psi(\sigma)$.
Choose now two arbitrary elements $s$, $s'$ in $S$. We want to prove that either $s$ is an initial segment of $s'$ or $s'$ an initial segment of $s$. Assume there is an index $i <  \min\{lh(s),lh(s')\}$ such that $(s)_i < (s')_i$ and $\forall k < i$, $(s)_k = (s')_k$. The case $(s)_i > (s')_i$ can be handled in a similar way. Note that $(s)_i < (s')_i$ is seen as an inequality between natural numbers and not between elements of $\T(\cdot^{*\dots * })$. Because $s$ is in $S$, $s$ can be extended to an infinite bad sequence $Y$ in $\T(\cdot^{*\dots*})$. This, however, contradicts $\psi(s')$ because $(Y)_k = (s)_k =  (s')_k$ for all $k<i$, but $(Y)_i = (s)_i < (s')_i$.\\
If $s\in S$, one can show by minimization in $RCA_0$, that there is a $z \in \T(\cdot^{*\dots *})$ such that $s^{\frown} (z) \in S$. Therefore, there exists an infinite sequence $(s_i)_{i< \omega}$ in $\T(\cdot^{*\dots *})$ such that $S = \{(s_0,\dots,s_i) : i < \omega\}$. 


\medskip

Now, define $subS$ as the set of all pairs $(i,t)$ such that $t$ is a leaf of $s_i$. Remark that $subS$ is definable in $RCA_0$, because
\begin{align*}
(i,t) \in subS & \Leftrightarrow \exists \sigma ( \sigma \in S  \text{ and } lh(\sigma) = i+1 \text{ and $t$ is a leave of $\sigma_i$})\\
& \Leftrightarrow \forall \sigma (( \sigma \in S \text{ and } lh(\sigma)=i+1 ) \rightarrow \text{$t$ is a leave of $\sigma_i$}). 
\end{align*}
On $subS$, we define the following ordering: $(i,t) \leq (j,t') \Leftrightarrow t \leq_{\T(\cdot^{*\dots*})} t'$. With this ordering $subS$ is a quasi-order. 
We want to prove that it is a well-quasi-order.

\medskip 

Assume that this is not true. This implies the existence of an infinite bad sequence $((n_i,s'_i))_{i<\omega}$ in $subS$. This implies $s'_i \leq s_{n_i}$ for all $i$.
Construct now an infinite subsequence $H$ such that $(n_i)_{i<\omega}$ is strictly increasing and the first element of $H$ is $(n_0,s'_0)$. So $H = \{(i,(n_{j_i},s'_{j_i})) : i<\omega\}$ with $j_0=0$.
This is possible in $RCA_0$, because the number of leaves of an element in $\T(\cdot^{*\dots*})$ is finite.\\
Construct now recursively a set $S'$ such that all elements that lie in $S'$ have the form $(\sigma_0,\dots,\sigma_i)$ such that if $i<n_0$, the element $\sigma_i$ is equal to $s_i$ and if $i\geq n_0$, then $\sigma_i$ is equal to $s'_{j_{i-n_0}}$. 
The existence of $S'$ is however in contradiction with the definition of $S$. This is because $\psi((s_0,\dots,s_{n_0}))$, $s'_{j_0} = s'_0 < s_{n_0}$ as a natural number and $(\sigma_i)_{i<\omega}$ is an infinite bad sequence. 
Take for example $k<l<\omega$. If $k<l<n_{0}$, then $\sigma_k \not\leq \sigma_l$ follows from $s_k \not\leq s_l$. If $n_{0}\leq k <l$, then $\sigma_k \not\leq \sigma_l$ follows from $s'_{j_{k-n_0}} \not\leq s'_{j_{l-n_0}}$. Assume $k < n_{0} \leq l$. Then $\sigma_k \not\leq \sigma_l$ follows from the fact that otherwise $s_k \leq  s'_{j_{l-n_{0}}} \leq s_{n_{j_{l-n_{0}}}}$ with $k<n_{0}=n_{j_0} \leq n_{j_{l-n_{0}}}$, a contradiction. So we conclude that $subS$ is indeed a well-quasi-order. 

\medskip

By Lemma \ref{star}, we obtain that $(subS)^{*\dots *}$ is a well-quasi-ordering. Now, look at the infinite sequence $(s_i)_{i<\omega}$ in $\T(\cdot^{*\dots*})$. Rewrite every element $\times s_i$ to an element in $(subS)^{*\dots *}$ and call it $\overline{s_i}$. For example, if $s_i=\circ\left(\left(s^1_1,\dots,s^1_{n_1}\right),\dots,\left(s^k_1,\dots,s^k_{n_k}\right)\right)$, then $\overline{s_i}$ is equal to $\left(\left((i,s^1_1),\dots,(i,s^1_{n_1})\right),\dots,\left((i,s^k_1),\dots,(i,s^k_{n_k})\right)\right)$. Because we know that $(subS)^{*\dots *}$ is a well-quasi-order, there exist two indices $i<j$ such that $\overline{s_i} \leq_{(subS)^{*\dots *}} \overline{s_j}$. Therefore, $s_i \leq_{\T(\cdot^{*\dots*})}s_j$, a contradiction.
 
\end{proof}

Remark that the previous proof can actually be simplified. We wrote it however in this way, so that the reader can see what is needed to prove the well-partial-orderedness of $\T(W)$: the proof can be carried out in $RCA_0 + (\Pi^1_1(\Pi^0_3)$--\,$CA_0)^-+ \forall X (X$ is a $\wpo \rightarrow  X^*$ is a $\wpo$). This theory is the same as $ACA_0 + (\Pi^1_1$--\,$CA_0)^-$.
A general guideline is that for a natural theory $T\supseteq RCA_0$ with $RCA_0 \vdash T \leftrightarrow \forall X (X$ is a $\wpo  \rightarrow W(X)$ is a $\wpo)$, the following holds
\[T+ (\Pi_1^1(\Pi^0_3)\text{--}\,CA_0)^- \vdash \T(W)\text{ is a $\wpo$}\]
and 
\[T+ (\Pi_1^1(\Pi^0_3)\text{--}\,CA_0)^- \not\vdash \T(W')\text{ is a $\wpo$},\]
where $W'$ is the `limit' of $W$. The notion `limit' is of course very vague. For example a binary tree in $\mathcal{B}(X)$ can be seen in some sense as an element of $X^{*\dots*}$ and a finite sequence in $X^*$ can be seen as an element of $X^n$. Therefore, $\mathcal{B}(\cdot)$ is the limit of $\cdot^{*\dots*}$ and $\cdot^*$ is the limit of $\cdot^n$. We do not give a specific definition of the notion `limit' as this is only a guideline and not a real theorem.

\subsection{Independence results for $RCA_0 + (\Pi^1_1(\Pi^0_3)$--\,$CA_0)^{-}$}\label{sec:independence results for RCA_0 + (Pi^1_1(Pi^0_3)--CA_0)^{-}}
We know that $\vert RCA_0 + (\Pi^1_1(\Pi^0_3)$--\,$CA_0)^{-}\vert = \vartheta(\Omega^\omega)$. Furthermore, one can easily show as in Theorem \ref{computation lower bound W(X)=B(X)} that
$RCA_0 \vdash \T(X^*)$ is a $\wpo$ $\rightarrow WF(\vartheta(\Omega^\omega))$. Therefore, $RCA_0 + (\Pi^1_1(\Pi^0_3)$--\,$CA_0)^{-} \not\vdash$ `$\T(X^*)$ is a $\wpo$'.

\medskip 

Again, as in the $ACA_0$-case, one can search for provability results: $\T(X^*)$ can be seen as the limit of $\T(X^n)$. Therefore, it would be interesting if $RCA_0 +(\Pi^1_1(\Pi^0_3)$--\,$CA_0)^{-} \vdash$ `$\T(X^n)$ is a $\wpo$'. However, the theory $RCA_0$ does not prove $\forall X$(`$X$ is a $\wpo$' $\rightarrow$ `$X^n$ is a $\wpo$') (see \cite{Reversemathematicsandtheequivalenceofdefinitionsforwellandbetterquasiorders} for more information). 
The theorem $\forall X$(`$X$ is a $\wpo$' $\rightarrow$ `$X^n$ is a $\wpo$') is however provable in $RCA_0 + CAC$, where $CAC$ is the principle saying that every infinite sequence in a partial order has a subsequence that is either a chain or an antichain.
This implies that $RCA_0 +CAC+(\Pi^1_1(\Pi^0_3)$--\,$CA_0)^{-} \vdash$ `$\T(X^n)$ is a $\wpo$'.

\renewcommand{\abstractname}{Acknowledgements}

\begin{abstract}
The first author wants to thank his funding organization Fellowship of the Research Foundation - Flanders (FWO). 
The second author's research was made possible through the support of a grant from the John Templeton Foundation. The opinions expressed in this publication are those of the authors and do not necessarily reflect the views of the John Templeton Foundation.
The authors also wish to thank Leszek Ko\l odziejczyk for his helpful comments and Ryota Akiyoshi for his fruitful discussions with the first author.
\end{abstract}

\bibliographystyle{spmpsci}      
\bibliography{Universal}

\end{document}